\documentclass[11pt,letterpaper]{amsart}

\usepackage{graphicx} % Required for inserting images
\usepackage{amsthm,amsmath,amssymb,amsfonts}
\usepackage{latexsym}
\usepackage{graphicx}
\usepackage{setspace}
\usepackage{verbatim}
\usepackage{rotating}
\usepackage{enumerate}
\usepackage{color}                       % For creating colored text and background
\usepackage[dvipsnames]{xcolor}
\definecolor{verde}{HTML}{288906}
\usepackage{hyperref}      % For creating hyperlinks in cross references

\usepackage{enumitem}
\usepackage{caption}
\usepackage{subcaption}
\usepackage{tikz-cd}
\usepackage{xfrac}
\usepackage{mathtools}
\usepackage[all]{xy}

\usepackage[capitalise]{cleveref}
\usepackage{tikz}
\usetikzlibrary{backgrounds}
\usetikzlibrary {shapes.geometric, arrows, arrows.meta}
\pgfdeclarelayer{bg}    
\pgfsetlayers{bg,main}
\usetikzlibrary{decorations.pathreplacing,calligraphy}
\usetikzlibrary{tikzmark}

\usepackage[utf8]{inputenc}
\usepackage{pgfplots}
\pgfplotsset{compat=1.11}
\usepgfplotslibrary{fillbetween}
\usetikzlibrary{intersections}
\pgfdeclarelayer{bg}
\pgfsetlayers{bg,main}
\usetikzlibrary{calc}
\definecolor{tempcolor}{RGB}{224,235,239}
\usepackage[utf8]{inputenc}
\usepackage{dirtytalk}
\usepackage{graphicx}
\usepackage{float}
\usepackage{multicol}

\usepackage{blindtext}
\usetikzlibrary{matrix,cd}
\usepackage[utf8]{inputenc}
\usepackage{xcolor}
\usepackage{mathtools}
\usepackage{listings}

\title[On the weak and strong Lefschetz properties for $\uppercase {R}/\init(\uppercase {I}_t)$]{On the weak and strong Lefschetz properties for initial ideals of determinantal ideals with respect to diagonal monomial orders}

\author{Hongmiao Yu}
\address{Vietnam Institute for Advanced Study in Mathematics, Hanoi, Vietnam}
\email{hongmiaoyu@hotmail.com}

\newtheoremstyle{break}
{}{}
{\slshape}{}
{\bfseries}{.}
{5pt}{}

\theoremstyle{definition}

\newtheorem{theorem}{Theorem}[section]
\newtheorem{lemma}[theorem]{Lemma}
\newtheorem{proposition}[theorem]{Proposition}
\newtheorem{corollary}[theorem]{Corollary}

\newtheorem*{theorem*}{Theorem}
\newtheorem*{theoremMain}{Main Theorem}
\newtheorem*{theorem1}{Theorem \ref{maintheorem}}

\newtheorem{definition}[theorem]{Definition}

\theoremstyle{remark}
\newtheorem{remark}[theorem]{Remark}
\newtheorem{example}[theorem]{Example}
\newtheorem{examples}[theorem]{Examples}
\newtheorem{notation}[theorem]{Notation}
\newtheorem{question}{Question}
\def\N{\mathbb N}

\def\R{\mathbb R}
\def\Z{\mathbb Z}
\def\mm{\mathfrak m}

\def\H{H} 

\def\se{\subseteq}

\def\HF{\mathrm {HF}}

\def\la{\longrightarrow}

\def\:{\colon}

\def\iso{\cong}
\def\Hom{\mathrm{Hom}}

\def\reg{\mathrm{reg}}

\def\Tor{\mathrm{Tor}}
\def\Ext{\mathrm{Ext}}

\def\soc{\mathrm{Soc}}

\def\mini{\mathrm{min}}

\def\init{\mathrm{in}}

\def\gin{\mathrm{gin}}
\def\lk{\mathrm{lk}}
\def\relint{\mathrm{relint}}
\def\conv{\mathrm{conv}}
\def\deg{\mathrm{deg}}
\def\h{k}

\def\cocoa{\mbox{\rm 
   C\kern-.13em o\kern-.07 em C\kern-.13em o\kern-.15em A}}
\newcommand\fg[0]{finitely generated{}}

\newcommand\V[1]{\mathrm{Vert}({#1})}

\begin{document}
\begin{abstract}
%Understanding the Lefschetz properties of Artinian reductions of Stanley–Reisner rings is an important problem from both algebraic and combinatorial perspectives. In this paper, we
%We study the Lefschetz properties for $R/\init(I_t)$, where $\init(I_t)$ is the initial ideal of the ideal $I_t$ of a polynomial ring $R$ with respect to a diagonal monomial order, and $I_t$ is generated by the $t$-minors of an $m\times n$  matrix of indeterminates.
%We show that if $t=\mini\{m,n\}$ then $R/\init(I_t)$ has the strong Lefschetz property. For each $t<\mini\{m,n\}$, we provide a bound such that  $R/\init(I_t)$ fails to satisfy the weak Lefschetz property when the product $mn$ exceeds this bound. As an application of this result, we present counterexamples that provide a negative answer to Murai's question on Lefschetz properties under square-free Gr\"obner degenerations.
%Moreover, we present (counter)examples of ideals $J$ of a standard graded polynomial ring $S$ such that its initial $\init(J)$ with respect to some monomial order is  square-free  and $S/J$ has the strong Lefschetz property, but $S/\init(J)$ fails the weak Lefschetz property.
We study the weak and strong Lefschetz properties for  $R/\init(I_t)$, where $I_t$ is the ideal of a polynomial ring $R$ generated by the $t$-minors of an $m\times n$  matrix of indeterminates, and $\init(I_t)$ denotes the initial ideal of  $I_t$ with respect to a diagonal monomial order.  We show that when $I_t$ is generated by maximal minors (that is, $t=\mini\{m,n\}$), the Stanley–Reisner ring $R/\init(I_t)$ has the strong Lefschetz property for all $m,n$. In contrast, for $t<\mini\{m,n\}$, we provide a bound such that  $R/\init(I_t)$ fails to satisfy the weak Lefschetz property whenever the product $mn$ exceeds this bound. As an application, we present counterexamples that provide a negative answer
to a question posed by Murai regarding the preservation  of Lefschetz properties under square-free Gr\"obner degenerations.

\end{abstract}

\maketitle

\section{Introduction} \label{intro}
The study of algebraic Lefschetz properties is motivated by the Hard Lefschetz Theorem in algebraic topology \cite{Le} (see also \cite{Hod}) and has become an important topic in commutative algebra, algebraic geometry, and combinatorics.
In general, the  weak and strong Lefschetz properties are considered over a standard graded Artinian algebra. Here we recall a generalized definition: 
Let $R$ be a standard graded polynomial ring, $I$ a homogeneous ideal of $R$, and let $d$ be the Krull dimension of $R/I$. We say that $R/I$ has the \textit{weak Lefschetz property (WLP)} if $R/I$ is Cohen–Macaulay and there exists a linear system of parameters $\underline\theta=\theta_1, \dots, \theta_d\in R_1$ of $R/I$ and a linear form  $ L \in R_1$ such that the multiplication map $$\times L: (R/(I, \underline\theta))_j\la (R/(I, \underline\theta))_{j+1}$$ has maximal rank for all $j$, that is, $\times L$ is either injective or surjective; we say that $A$ has the \textit{strong Lefschetz property (SLP)}  if the multiplication map $$\times L^s: (R/(I, \underline\theta))_j\la (R/(I, \underline\theta))_{j+s}$$ has maximal rank for all $j$ and for all $s$. A linear form $L\in R_1$ for which the multiplication by $L$ has maximal rank in all degrees is called
 a \textit{weak (resp. strong) Lefschetz element} for $R/I$.
%Studying the weak and strong Lefschetz properties for Stanley-Reisner rings is an important problem from both algebraic and combinatorial perspectives. 

An especially noteworthy result regarding the weak and strong Lefschetz properties under Gr\"obner degenerations is the following: 
%\begin{quote}
%{\it Let $R$ be a standard graded polynomial ring over an infinity field $K$, $I$ a homogeneous ideal of $R$ such that $\dim R/I= d$, and let $\init_<(I)$ be the initial ideal of $I$ with respect to a term order $<$.  If $R/\init_<(I)$ has the WLP (resp. SLP), then $R/I$ has the WLP (resp. SLP). }
%\end{quote}
\begin{lemma}[Wiebe, Murai]\label{WiebeMurai}
Let $R$ be a standard graded polynomial ring over an infinity field $K$, $I$ a homogeneous ideal of $R$ such that $\dim R/I= d$, and let $\init_<(I)$ be the initial ideal of $I$ with respect to a monomial order $<$.  If $R/\init_<(I)$ has the WLP (resp. SLP), then $R/I$ has the WLP (resp. SLP). 
\end{lemma}
This result was first proved by Wiebe in 2004 for the case $d=0$ \cite[Proposition 2.9]{Wi}, and was later generalized to all $d\ge0$ by Murai \cite[Lemma 3.3]{Mu}. However, the converse does not hold in general. Therefore, in light of Conca and Varbaro's result \cite[Corollary 2.11 (iii)]{CV}  which states that
%, asserting that
  if $\init_<(I)$ is a square-free monomial ideal  then $R/I$ is Cohen–Macaulay if and only if $R/\init_<(I)$ is Cohen–Macaulay, in a private conversation, Murai posed the following question:
\begin{question}[Murai]\label{muquestion}
Suppose that $\init_{<}(I)$ is a square-free monomial ideal and that $R/I$ has the WLP (resp. SLP). Does it follow that $R/\init_{<}(I)$ also has the WLP (resp. SLP)?
%If $\init_<(I)$ is a square-free monomial ideal and $R/I$ has the WLP (resp. SLP), can we have that $R/\init_<(I)$ has the WLP (resp. SLP)?
\end{question}
The motivation for this paper stems from this question. In addressing it, we study the Lefschetz properties for a particular Stanley–Reisner ring:
Let $K$ be a  field of characteristic zero,  $X=(X_{i,j})_{1\le i\le m, 1\le j\le n}$ an $m\times n$ matrix of indeterminates and let $R= K[X]$ be a standard graded polynomial ring over $K$. For a  positive integer $t$ such that $2\le t\le \mini\{m,n\}$ and $t<\max\{m,n\}$, denote by $I_t$  the ideal of $R$ generated by the $t$-minors of $X$, and  denote by $\init(I_t)$  the initial ideal  of  $I_t$ with respect to a diagonal monomial order.
%, that is, a monomial order under which the initial monomial of a minor is the product of the indeterminates along its main diagonal.  
%For example, the  (degree) lexicographic order and the (degree) reverse lexicographic on $R$ induced by the ordering  $X_{1,1}>X_{1,2}>\dots>X_{1,n}> X_{2,1}>\dots>X_{m,n}$ are diagonal monomial orders.
These initial ideals are known to be square-free (see \cite[Corollary 3.4]{Na})  and their corresponding quotient rings $R/\init(I_t)$ are 
%It is known that the initial ideal $\init(I_t)$ of  $I_t$ with respect to a diagonal monomial order is square-free \cite[Corollary 3.4]{Na} (see also \cite[Theorem 5.3 and Remark 4.7(c)]{BC2}) and its corresponding quotient ring $R/\init(I_t)$ is 
Cohen–Macaulay  (see, for example, \cite[Theorem 4.4.5]{BCRV}). 
In this paper, we show that
%Given that $R/\init(I_t)$ is Cohen-Macaulay for all $t,m,n$ (see, for example, \cite[Theorem 4.4.5]{BCRV}), we show in this paper  that 
\begin{theoremMain}
 If $t=\mini\{m,n\}$,  then $R/\init(I_t)$ has the SLP for all $m,n$;\\
 if $t<\mini\{m,n\}$,  then $R/\init(I_t)$ fails the WLP when $t,m,n$ satisfy one of the following conditions:
\begin{itemize}
\item[$i$)] $t=2$ and $mn\ge 16$,
\item[$ii$)] $t=3$ and $mn\ge 24$,
\item[$iii$)] $t\ge4$ and $mn\ge (t+1)(t+2)$.
\end{itemize}
\end{theoremMain}
This result will be proved through  \Cref{casem=n} and \Cref{lasttheorem}. A crucial step in proving \Cref{lasttheorem} is the following statement, which is the main theorem  of \Cref{sectionmain}:
\begin{theorem}\label{maintheorem}
For all $t$ such that $2\le t \le \mini\{m,n\}$ and $t< \max\{m,n\}$,  the graded Betti number $$\beta_{h, h+t-1}(R/\init(I_t))\ge t,$$ where $h=(m-t+1)(n-t+1)$ is the height of $\init(I_t)$.
\end{theorem}
Additionally, as an application of the Main Theorem, in  \Cref{answer} we answer Murai’s question in the negative  by showing that
\begin{center}
\textit{For $m=n\ge t+2$, $R/\init(I_t)$ fails the WLP while $R/I_t$ has the SLP.}
\end{center} 
%(see \Cref{answer}.)\\

The rest of this paper is organized as follows: \Cref{sectionpre} collects basic definitions and results that will be used throughout the paper. In particular,  in this section we introduce and study a special subcomplex  $\Omega_a(t, m,n)$ of the simplicial complex defined by $\init(I_t)$. The goal of  \Cref{sectionmain} is to prove \Cref{maintheorem} by using Hochster's formula and some properties of $\Omega_a(t, m,n)$. \Cref{lef} contains detailed discussion on the Lefschetz properties for $R/\init(I_t)$ and the other results mentioned above. Several calculations were carried out using the computer algebra system Macaulay2 \cite{M2}, and the corresponding code is provided in \Cref{M2code}.

\section{Preliminaries}\label{sectionpre}
\subsection{Some results on the weak and strong Lefschetz properties} 
The following lemma is simple, but it illustrates the connection between Lefschetz properties and graded Betti numbers, which is crucial for understanding the main idea of the paper.
\begin{lemma}\label{rmkapp}
Let  $S=K[X_1,\dots, X_N]$ be a standard graded polynomial ring,  $J\se S$ a homogeneous ideal of height $h$ such that  $S/J$ is Cohen-Macaulay, and let $\underline\theta=\theta_1,\dots, \theta_{N-h}\in S_1$ be  a linear system of parameters of $S/J$.
%a sequence of general linear forms. 
If $\beta_{h,h+j}(S/J)\not=0$ for some $j\ge 0$, then the multiplication map $$\times L^s: \Big[S/(J, \underline\theta)\Big]_{j}\la \Big[S/(J, \underline\theta)\Big]_{j+s}$$ fails to be injective for every linear form $L$ and for every $s\ge 1$.\end{lemma}
\begin{proof}
First recall that for a graded $S$-module $M$, the \textit{socle of $M$} is defined as
\begin{align*}
\soc(M)=0:_M \mm \iso\Hom_S(K, M),
\end{align*}
where $\mm=(X_1,\dots, X_N)$ is the unique homogeneous maximal ideal of $S$.  
%%If $J\se S$ is  a homogeneous  ideal  and if $\underline l= l_1, \dots, l_{c}\in S_1$ is a sequence of general linear forms with $c\le \dim S/J$, then the following graded isomorphisms hold
%\begin{align*}
%\soc(S/(J, \underline l))
%\iso \Hom_S(K, S/(J, \underline l))
%\iso \Hom_{S/\underline l}(K, S/(J, \underline l))\\
%\iso  \Ext^{c}_S(K, S/J)[-c]
%\iso \Tor_{N-c}^S(K, S/J)[N-c],
%\end{align*} 
%and
%\begin{align*}
%\soc(S/(J, \underline\theta))\iso \Hom_S(K, S/(J, \underline\theta))
%\iso \Tor_{N}^S(K, S/(J, \underline\theta))[N]
%\end{align*}
It is known that the following graded isomorphisms hold
\begin{align*}
\soc(S/(J, \underline \theta))
\iso& \Hom_S(K, S/(J, \underline \theta))
\iso \Hom_{S/\underline \theta}(K, S/(J, \underline \theta))\\
\iso  &\Ext^{N-h}_S(K, S/J)[-N+h]
\iso \Tor_{h}^S(K, S/J)[h]
\end{align*} (see, for example, \cite[Lemma 3.1.16 and Exercise 3.3.26]{BH}, see also \cite[Proposition 1.6.9 and Proposition 1.6.10]{BH}).
Therefore, %we get the following   equalities, which will be used later in this section:
%\begin{align}
%\beta_{N-c-1, N-c-1+j}(J)&=\beta_{N-c,N-c-1+j}(S/J) \notag\\
%&=\dim_K \Big[\Tor_{N-c}^S(K, S/J)\Big]_{N-c-1+j} \notag\\
%&=\dim_K \Big[\soc(S/(J, \underline\theta))\Big]_{j-1}.\label{socle1}
%\end{align}
\begin{align*}
\dim_K \Big[\soc(S/(J, \underline \theta))\Big]_{j}=\dim_K \Big[\Tor_{h}^S(K, S/J)\Big]_{h+j} =\beta_{h,h+j}(S/J)
%=\beta_{N-c-1, N-c+j}(J).
\end{align*}
 for each $j$, the last equality above follows from \cite[Proposition 1.3.1]{BH}. Thus, by the definition of socle,  if $\beta_{h,h+j}(S/J)\not=0$ then $\ker(\times L^s)\not=0$ for every linear form $L$ and for every $s\ge 1$.
%and
%\begin{align}
%\beta_{N-c-1, N-c-1+j}(J)=\beta_{N-1, N-1+j}(J+\underline\theta).\label{socle}
%\end{align}\
\end{proof}

We now recall the following well-known result  (see, for example,  \cite[Theorem 4.2]{Swa} and \cite[Lemma 3.1]{Mu}).
\begin{lemma}\label{openset}
Let  $S=K[X_1,\dots, X_N]$ be a standard graded polynomial ring over a field of characteristic zero,  $J\se S$ a homogeneous ideal such that $\dim S/J=d$.   If $S/J$ has the WLP (resp. SLP), then there exits a nonempty Zariski open subset $U\se K^{n\times(d+1)}$ such that, for any sequence of linear forms $\theta_1, \dots, \theta_d, L\in U$, we have $\theta_1, \dots, \theta_d$ is a linear system of parameters of $S/J$ and $L$ is a weak (resp. strong) Lefschetz element of $S/J$.
\end{lemma}
Assume for the remainder of this paper that $K$ is a field of characteristic zero. According to \Cref{openset}, to verify the Lefschetz properties for a $d$-dimensional Cohen–Macaulay ring 
$S/J$, it suffices to consider a sequence of general linear forms $\theta_1, \dots, \theta_d, L$, and check whether the multiplication map defined by $L$ on $S/(J, \theta_1, \dots, \theta_d)$ has maximal rank in each degree. 
Therefore, \Cref{rmkapp} and  \Cref{openset} imply that

\begin{corollary}\label{corormkapp}
Let  $S=K[X_1,\dots, X_N]$ be a standard graded polynomial ring,  $J\se S$ a homogeneous ideal of height $h$. If $\beta_{h,h+j}(S/J)\not=0$ for some $j\ge 0$ and, for a sequence of general linear forms  $\underline\theta=\theta_1, \dots, \theta_{N-h}$, the multiplication map $$\times L: \Big[S/(J, \underline\theta)\Big]_{j}\la \Big[S/(J, \underline\theta)\Big]_{j+1}$$ fails to be surjective for any linear form $L$, then $S/J$ fails the WLP.
\end{corollary}

Recall that if $M=\oplus_{i\in\Z}M_i$ is a \fg\ graded module over a polynomial ring $S=K[X_1,\dots, X_N]$, then the \textit{Hilbert function of $M$} is  the function $$\HF(M, -): \N\la \N$$ defined by $\HF(M, i)=\dim_K M_i.$
Macaulay’s theorem \cite{Ma} (see also \cite[Section 4.2]{BH}) ensures that an ideal and its initial ideal have the same Hilbert function. 
Moreover, for a  linear form $L$ and for an integer $s\ge 1$, the following sequence
\begin{align*}
    0\la \Bigg[\frac{(J, \underline\theta): L^s}{(J, \underline\theta)}\Bigg]_{j}\la  \Big[S/(J, \underline\theta)\Big]_{j}&\stackrel {\times L^s}\la \Big[S/(J, \underline\theta)\Big]_{j+s}\\
    &\la \Big[S/(J, \underline\theta, L^s)\Big]_{j+s}\la 0\notag
\end{align*}
is exact in each  degree $j\ge 0$. Therefore, 
\begin{align}\label{hfses}
\HF(S/(J, \underline\theta), j+s)-\HF(S/(J, \underline\theta),j)
+\HF(\frac{(J, \underline\theta): L^s}{(J, \underline\theta)}, j)&\\
=\HF(S/(J, \underline\theta, L^s),j+s).&\notag
\end{align}

A slight extension of the argument in Conca’s proof of \cite[Lemma 1.2]{Co} (see also the proof of \cite[Proposition 2.8]{Wi} for a similar discussion) yields the following statement.

\begin{lemma}[Conca]\label{ginconca}
Let  $S=K[X_1,\dots, X_N]$ be a polynomial ring,  and let $J$ be a homogeneous ideal. Let $p$ be an integer with $0\le p\le N$, and let $\theta_1,\dots, \theta_p\in S_1$ be a sequence of $p$ general linear forms. Then 
$$\HF(S/(J, \theta_1^s, \theta_2,\dots, \theta_p), j)=\HF(S/(\gin(J), X_{N-p+1}^s, X_{N-p+2},\dots, X_{N}), j)$$
for all $j\in\N$ and for all $s\ge 1$, where $\gin(J)$ denotes the generic initial ideal of $J$ with respect to the reverse lexicographic order.
\end{lemma}
%
%A slight extension of the argument in Conca’s proof of \cite[Lemma 1,2]{Co} implies the following statement:
%\begin{corollary}\label{ginconca2}
%Let  $S=K[X_1,\dots, X_N]$ be a polynomial ring,  $J\se S$ a homogeneous ideal of height $h$, and let $\theta_1,\dots, \theta_{N-h}, L\in S_1$ be  a sequence of general linear forms. Then 
%$$\HF(S/(J, \underline\theta),j)=\HF(S/(\gin(J), X_{h+1}, \dots, X_N),j),$$
%$$\HF(S/(J, \underline\theta, L^s),j)=\HF(S/(\gin(J), X_{h}^s, X_{h+1}, \dots, X_N),j),$$
%for all $j$ and for all $s\le \indeg(J)$, where ${\indeg(J)}=\mini\{i\mid [J]_i\not=0\}$ is the initial degree of the ideal $J$.
%\end{corollary}

The following remark follows directly from Macaulay’s theorem, \Cref{openset} and  \Cref{hfses}. The Macaulay2 code in \Cref{M2code}, which verifies the weak and strong Lefschetz properties for $R/\init(I_t)$ and for $R/I_t$, is written based on this remark and  on \Cref{WiebeMurai}.

\begin{remark}\label{rmk1.5}
Let  $S=K[X_1,\dots, X_N]$ be a standard graded polynomial ring, and  let $J$ be an ideal of $S$ of height $h$, generated by polynomials of degree at least $2$. Assume that $S/J$ is Cohen-Macaulay. Set $A=K[X_{N-h+1},\dots, X_N]$.
If  $\underline\theta=\theta_1, \dots, \theta_{N-h}\in S_1$ is a sequence of general linear forms and if  $<$ is  the  (degree) lexicographic order or the (degree) reverse lexicographic on  $S$ induced by the ordering $X_1>X_2>\dots>X_N$,
then $$\HF\big(S/(J, \underline\theta), j\big)=\HF\Big(A/\big(\init_<(J, \underline\theta)\cap A\big), j\Big)$$
for all $j\ge 0$.
Moreover, $S/J$ has the WLP if and only if there exists  
%a linear system of parameters $\underline\theta=\theta_1, \dots, \theta_{N-h}\in S_1$ of $S/J$ and 
a linear form  $ L \in S_1$ such that 
\begin{align*}
&\max\Big\{0,\HF\Big(A/\big(\init_<(J, \underline\theta)\cap A\big),j+1\Big)-\HF\Big(A/\big(\init_<(J, \underline\theta)\cap A\big),j\Big)\Big\}\\
&=\HF\Big(A/\big(\init_<(J, \underline\theta, L)\cap A\big),j+1\Big)
\end{align*}
for all $j\ge 0$; $S/J$ has the SLP if and only if 
\begin{align*}
&\max\Big\{0, \HF\Big(A/\big(\init_<(J, \underline\theta)\cap A\big),j+s\Big)-\HF\Big(A/\big(\init_<(J, \underline\theta)\cap A\big),j\Big)\Big\}\\
&=\HF\Big(A/\big(\init_<(J, \underline\theta, L^s)\cap A\big),j+s\Big)
\end{align*}
for all $j\ge 0$ and for all $s\ge 1$.
\end{remark}

%\subsection{Simplicial complex $\Omega_a(t, m,n)$}
\subsection{Simplicial complexes and Hochster's formula} 
%The aim of the remainder of this section is to introduce simplicial complex $\Omega_a(t, m,n)$ and show how it relates to $\beta_{h,h+t-1}(R/\init(I_t))$.

Let $X$ be an $m\times n$ matrix of
indeterminates. Recall that a monomial order on $R=K[X]$ is said to be a \textit{diagonal monomial order} if  the initial monomial of any minor of $X$ is the product of the indeterminates along its main diagonal.  
For example, the  (degree) lexicographic order and the (degree) reverse lexicographic on $R$ induced by the ordering  $X_{1,1}>X_{1,2}>\dots>X_{1,n}> X_{2,1}>\dots>X_{m,n}$ are diagonal monomial orders.

Narasimhan proved that the $t$-minors of  $X$  form a Gr\"obner basis of $I_t$ with respect to a diagonal monomial order \cite[Corollary 3.4]{Na} (for other proofs see also  \cite[Theorem 5.3 and Remark 4.7(c)]{BC2}, \cite{CGG}, and \cite[Theorem 1]{Stu}). Therefore,     the initial ideal  $\init(I_t)$ of  $I_t$ with respect to a diagonal monomial order is a square-free monomial ideal for all $t,m,n$. By the Stanley–Reisner correspondence (see, for example, \cite[section 2.1]{BCRV}), there exists a simplicial complex $\Delta(t,m,n)$ such that the associated Stanley–Reisner  ideal $I_{\Delta(t,m,n)}=\init(I_t)$. This simplicial complex will be described in more detail in \Cref{7441}. Moreover, since  the graded Betti numbers of a Stanley-Reisner ideal can be obtained using Hochster's formula \cite{Ho} (see also \cite[Corollary 5.12]{MS} and \cite[Theorem 5.5.1]{BH}), here we recall the definition of simplicial complexes and Hochster’s formula.
%  for the graded Betti numbers of a Stanley-Reisner ideal. 

\begin{definition}
A \textit{simplicial complex} $\Delta$ on a finite set $V=\{1,\dots, N\}$ is a collection of subsets of  $V$ such that 
\begin{center}
if $F\in\Delta$ and $G\se F$, then $G\in\Delta$.
\end{center}
The elements of $\Delta$ are called \textit{faces}. The maximal faces under inclusion are called the \textit{facets} of the simplicial complex $\Delta$.
The \textit{dimension} of a face $F\in\Delta$ is defined as $\dim F=\vert F\vert -1$ and the \textit{dimension} of the simplicial complex $\Delta$ is $$\dim\Delta=\max\{\dim F\mid F\in\Delta\}.$$ 
\end{definition}

\begin{definition}\label{SRring}
Let $\Delta$ be a simplicial complex on the vertex set $V=\{1,\dots ,N\}$.
The \textit{Stanley-Reisner ring} of $\Delta$  (with respect to a filed $k$) is the homogeneous $k$-algebra 
$$k[\Delta]=k[X_1,\dots, X_N]/I_\Delta,$$
where $I_\Delta$ is the ideal generated by all monomials $X_{i_1}\dots X_{i_r}$  such that $\{i_1,\dots, i_r\}\not\in\Delta$. The ideal $I_\Delta$ is called  the \textit{Stanley-Reisner ideal} of $\Delta$.

\end{definition}

%Here we recall Hochster's formula for the graded Betti numbers of a Stanley-Reisner ideal \cite{Ho} (see also \cite[Corollary 5.12]{MS} and \cite[Theorem 5.5.1]{BH}), which will be used repeatedly in this section.
\begin{lemma}[Hochster's formula]\label{hoch}
Let $\Delta$ be a simplicial complex on the vertex set $V=\{1,\dots ,N\}$ and let $I_\Delta$ be the associated Stanley-Reisner ideal in $k[X_1,\dots,X_N]$. Then 
$$\beta_{i,j}(I_\Delta)=\sum_{\substack{\vert U\vert=j\\ U\se V}}\dim_K\Tilde{H}_{j-i-2}(\Delta_U;k)$$ for each $i,j$, 
where $\Delta_U=\{\sigma\in \Delta\mid \sigma\se U\}$ is the restriction of $\Delta$ to $U$, and $\Tilde{H}_l(\Delta_U; k)$ denotes the $l$-th reduced simplicial homology of $\Delta_U$. 
\end{lemma}

%The results below are well-known and can be found in  \cite[Theorem 5.3.2]{BH},\cite[Lemma 5.4.5]{BH}, \cite[Theorem 5.2.14]{BH}, \cite[Theorem 2.9]{BW}, \cite[Theorem 5.1.13]{BH}. We include them here, as they will be used in the proof of \Cref{mainlemma}.

\subsection{Simplicial complex $\Omega_a(t, m,n)$}\label{subsec3} The aim of the remainder of this section is to introduce simplicial complex $\Omega_a(t, m,n)$ and show how it relates to $\beta_{h,h+t-1}(R/\init(I_t))$.
In what follows, let $V=\{1, \dots, m\}\times\{1,\dots, n\}$  and let $\le$ be the  partial order on $V$ defined as 
%Let $\le$ be the partial order on $V$ defined as 
\begin{center}
 $(a, b)\le (c,d)$ if $a\le c$ and $b\ge d$.
\end{center}
%Recall the following definitions:

\begin{definition}
 A subset $W$ of $V$ is said to be a \textit{chain} (in the sense of  \cite[Section 4.1]{BCRV}) if each two elements of $W$ are comparable in the poset $(V,\le)$.\\
If  $(a, b)\le (c,d)$,  a \textit{path} $\mathcal P$ in  $V$ from  $(a,b)$ to  $(c,d)$  is an unrefinable chain with minimum $(a,b)$ and maximum $(c,d)$. 
That is,
$$\mathcal P=\{(a_1,b_1),\dots, (a_s, b_s)\}\se V,$$
%It can be considered as a sequence of points
%in $V$
%$$(a,b)=(a_1,b_1),\dots, (a_s, b_s)=(c,d),$$
where $(a_1,b_1)=(a,b)$,  $(a_s, b_s)=(c,d)$, and $$(a_{i+1}, b_{i+1})-(a_{i}, b_{i})\in\{(1,0), (0,-1)\}$$ for all $1\le i\le s-1$.
\end{definition}

\begin{definition}
Given two  sequences $\mathcal S_1= p_1, \dots, p_s $ and $\mathcal S_2= q_1, \dots, q_s$ of $s$ points in $V$, a \textit{family of nonintersecting paths from $\mathcal S_1$ to  $\mathcal S_2$}  (in the sense of  \cite[Section 4.4]{BCRV}) is a set $\mathcal F\se V$ such that
$$\mathcal F={\mathcal P_1}\cup {\mathcal P_2}\cup\dots\cup{\mathcal P_s},$$
where $\mathcal P_i$ is a path  from $p_i$ to $q_i$ for each $i$ and  ${\mathcal P_i}\cap { \mathcal P_j}=\emptyset$ if $i\not=j$.
\end{definition}

\begin{remark}\label{7441}
For each $t,m,n\in\N$  such that $2\le t\le \mini\{m,n\}$ and $t<\max\{m,n\}$, denote by $\Delta(t,m,n)$ the simplicial complex defined by $\init(I_t)$.  By \cite[Proposition 4.4.1]{BCRV}, the facets of $\Delta(t,m,n)$ are  the families of nonintersecting paths  from $(1,n)$, $(2,n)$, $\dots$, $(t-1,n)$ to $(m,1)$, $(m,2)$, $\dots$, $(m,t-1)$.
In particular, if $t=2$, then the facets of $\Delta(2,m,n)$ are the maximal chains from $(1,n)$ to $(m,1)$.
\end{remark}
\begin{definition}
For each integer $0\le a\le t-1$, let  $V_a(t, m,n)$ be the subset of $V$ defined by
\begin{align*}
V_a(t, m,n)=  \{(i,i)\mid 1\le i\le a \}\cup
\{(m-i,n-i) \mid 0\le i\le t-a-2 \} \cup &\\
\{(i,j) \mid a+1\le i \le m+a+1-t, a+1\le j \le n+a+1-t \}&
\end{align*}
which can be represented in Cartesian coordinates as shown in \Cref{omegatmn}: 
\begin{figure}[H]
    \centering
    \begin{comment}
   
%%%%%%%%%%%%%%%%%%%%%%%%%%%%%%%%%%%%%%%%%%%%%%%first draw
     \begin{tikzpicture}[scale=0.5]
 \foreach\x in{1,...,9}{
 \foreach\y in{1,...,7}{
 %\ifnum\x>\y
 \fill (\x,\y)circle(0.05);
 %\else
 %\fill[blue](\x,\y)circle(0.1);
 %\fi
 }
 }
 \foreach\x in{1,...,3}{
 \foreach\y in{1,...,3}{
  \ifnum\x=\y, 
   \fill (\x,\y)circle(0.08);
   \fi
 }
 }
  \foreach\x in{4,...,7}{
 \foreach\y in{4,...,5}{
 % \ifnum\x=\y, 
   \fill (\x,\y)circle(0.08);
%   \fi
 }
 }
  \foreach\x in{8}{
 \foreach\y in{6}{
 % \ifnum\x=\y, 
   \fill (\x,\y)circle(0.08);
%   \fi
 }
 }
  \foreach\x in{9}{
 \foreach\y in{7}{
 % \ifnum\x=\y, 
   \fill (\x,\y)circle(0.08);
%   \fi
 }
 }
 \node[label=below left:{$(1,1)$}] at (1,1) {};
 \node[label=above right:{$(m,n)$}] at (9,7) {};  
\node[label=below right:{$(m,1)$}] at (9,1) {}; 
\node[label=above left:{$(1,n)$}] at (1,7) {}; 
\node[label=above:{$(a,a)$}] at (3,3) {}; 
\draw[black!60, very thick] (1,0.7)--(0.7,1)-- (3,3.3)--(3.3,3)--cycle;
\draw[black!60, very thick] (7.95,5.65)--(7.7,6)-- (9,7.3)--(9.3,7)--cycle;
\draw[black!60, very thick] (3.8,3.8) rectangle (7.2,5.2);
 \end{tikzpicture}
 %%%%%%%%%%%%%%%%%%%%%%%%%%%%%%%%%%%%%%%%%%%%%%%first draw fin
 \end{comment}
 %%%%%%%%%%%%%%%%%%%%%%%%%%%%%%%%%%%%%%%%%%%%%%%second draw
 \begin{tikzpicture}[roundnode/.style={circle, draw=gray!80}, very thick, scale=0.7]
     \draw[very thin,gray!70]  (0.8,0.8) grid (9.2,8.2);
 \foreach\x in{1,...,3}{
 \foreach\y in{1,...,3}{
  \ifnum\x=\y, 
   \fill (\x,\y)circle(0.13);
   \fi
 }
 }
  \foreach\x in{4,...,7}{
 \foreach\y in{4,...,6}{
 % \ifnum\x=\y, 
   \fill (\x,\y)circle(0.13);
%   \fi
 }
 }
  \foreach\x in{8}{
 \foreach\y in{7}{
 % \ifnum\x=\y, 
   \fill (\x,\y)circle(0.13);
%   \fi
 }
 }
  \foreach\x in{9}{
 \foreach\y in{8}{
 % \ifnum\x=\y, 
   \fill (\x,\y)circle(0.13);
%   \fi
 }
 }
 \node[roundnode]at  (3,3) {};
 \node[label=  left:{$(1,1)$}] at (1,1) {};
 \node[label=  right:{$(m,n)$}] at (9,8) {};  
\node[label=  right:{$(m,1)$}] at (9,1) {}; 
\node[label= left:{$(1,n)$}] at (1,8) {}; 
\node[label=above:{$(a,a)$}] at (3,3) {}; 
\end{tikzpicture}
%%%%%%%%%%%%%%%%%%%%%%%%%%%%%%%%%%%%%%%%%%%%seconddrawfin
\caption{The set $V_a(t, m,n)$}
    \label{omegatmn}
\end{figure}
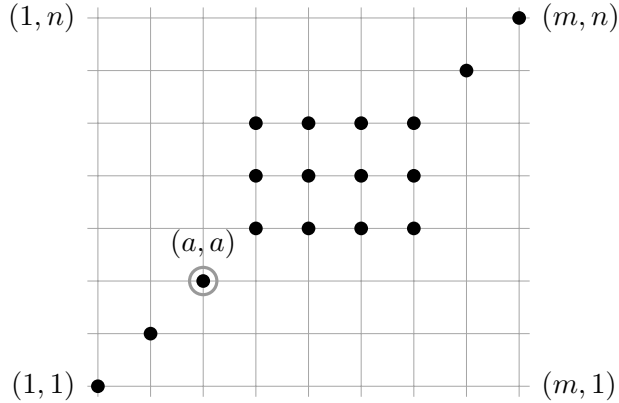
\end{definition}
\begin{definition}
For each integer $0\le a\le t-1$,
$$\Omega_a(t, m,n) =\{\sigma\in\Delta(t, m,n) \mid \sigma\se V_a(t, m,n)\}$$ is
%denote by $\Omega_a(t, m,n)$ 
the simplicial complex defined as the restriction of $\Delta(t, m,n)$ to the set $V_a(t, m,n)$.
\end{definition}
%We now discuss some properties of $\Omega_a(t, m,n)$. 

%That is, $\Omega_a(t, m,n) =\{\sigma\in\Delta(t, m,n) \mid \sigma\se  \V{\Omega_a(t, m,n)}\}$.

\begin{example}
If $t=2$, $m=n=3$, then the vertex set of  $\Omega_0(2,3,3)$ is $V_0(2,3,3)=\{(1,1), (1,2), (2,1), (2,2), (3,3)\}$:
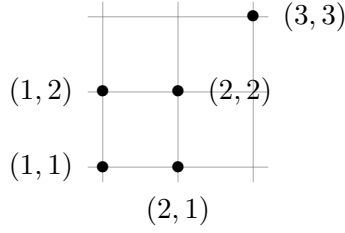
\begin{figure}[H]
    \centering
 \begin{tikzpicture}[roundnode/.style={circle, draw=gray!80}, very thick, scale=1]
     \draw[very thin,gray!70]  (0.8,0.8) grid (3.2,3.2);
% \node at (1,2){$\bullet$};
% \node at (2,2){$\bullet$};
% \node at (2,1){$\bullet$};
 \node[label=  left:{$(1,1)$}] at (1,1){$\bullet$};
 \node[label=  right:{$(3,3)$}] at (3,3) {$\bullet$};  
\node[label= below:{$(2,1)$}] at (2,1) {$\bullet$}; 
\node[label=  right:{$(2,2)$}] at (2,2) {$\bullet$}; 
\node[label=  left:{$(1,2)$}] at (1,2) {$\bullet$}; 
\end{tikzpicture}
\caption{The set $V_0(2,3,3)$}
\end{figure}
Moreover, by \Cref{7441}, the faces of  $\Omega_0(2,3,3)$ are the restrictions to the set  $V_0(2,3,3)$ of the chains from $(1,3)$ to $(3,1)$. For instance, since $\{(1,3), (1,2),  (1,1), (2,1), 
(3,1)\}$ is a  chain  from $(1,3)$ to $(3,1)$, it follows that $\{(1,2), (1,1), 
(2,1)\}$ is a face of  $\Omega_0(2,3,3)$. Furthermore,  the chain $\{(1,2), (1,1), 
(2,1)\}$ is maximal in $V_0(2,3,3)$, and so it is a facet of $\Omega_0(2,3,3)$. Similarly,  $\{(3,3)\}$ and $\{(1,2), (2,2), (2,1)\}$ are also facets of $\Omega_0(2,3,3)$.
Hence the simplicial complex $\Omega_0(2,3,3)$ can be represented as shown in \Cref{w0233}. 
%By \Cref{small}, $\vert\Omega_0(2,3,3)\vert=5$ and $\dim \Omega_0(2,3,3)=2.$
\begin{figure}[H]
    \centering
    \begin{tikzpicture}[scale=1]
\filldraw[color=black, fill=black!15, very thick](0.5,1.5)--(0,0)--(2,0)--cycle;
\filldraw[color=black, fill=black!15, very thick] (0.5,1.5)--(2.2,2)--(2,0);
\node[label=left:{$(1,1)$}] at (0,0) {$\bullet$};
\node[label=above:{$(1,2)$}] at (0.5,1.5) {$\bullet$};
\node[label=right:{$(2,1)$}] at (2,0) {$\bullet$};
\node[label=right:{$(2,2)$}] at (2.2,2) {$\bullet$};
\node[label=right:{$(3,3)$}] at (4,1) {$\bullet$};
\end{tikzpicture}
\caption{$\Omega_0(2,3,3)$}
    \label{w0233}
\end{figure}
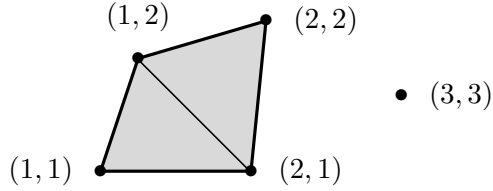
%\begin{figure}[H]
%    \centering
%    \begin{tikzpicture}[scale=1]
%\filldraw[color=black, fill=black!40, very thick](0,2)--(-0.7,0.3)--(0,0)--cycle;
%\filldraw[color=black, fill=black!15, very thick] (0,2)--(1.3,0.5)--(0,0);
%\node[label=left:{$(1,1)$}] at (-0.7,0.3) {$\bullet$};
%\node[label=above:{$(1,2)$}] at (0,2) {$\bullet$};
%\node[label=below:{$(2,1)$}] at (0,0) {$\bullet$};
%\node[label=right:{$(2,2)$}] at (1.3,0.5) {$\bullet$};
%\node[label=right:{$(3,3)$}] at (3,1.5) {$\bullet$};
%\end{tikzpicture}
%\caption{$\Omega_0(2,3,3)$}
%    \label{w0233}
%\end{figure}
\end{example}

\begin{example}
If $t=3$, $m=4$, $n=5$ and $a=1$, then
   the vertex set of  $\Omega_1(3,4,5)$ is  $$V_1(3,4,5)=\{(1,1), (2,2),  (2,3), (2,4), (3,2), (3,3), (3,4), (4,5)\}.$$ As shown in \Cref{v1345}, the set 
%\{(1,4), (1,5), (2,2), (2,3), (2,4), (2,5), (3,1), (3,2), (3,5), (4,2), (4,3), (4,4), (4,5)\}=
\begin{align*}
\mathcal F=\{ (1,5), (1,4), (2,4), (2,3), (2,2), (2,1), (3,1), (4,1)\}\cup \{(2,5), (3,5), &\\(4,5), (4,4), (4,3), (4,2)\}&
\end{align*} is a family of nonintersecting paths from $(1,5), (2,5)$ to $(4,1), (4,2)$, and its restriction to the set $V_1(3,4,5)$  is 
$\{ (2,2), (2,3), (2,4),  (4,5)\}$. Therefore, $\{ (2,2), (2,3), (2,4),  (4,5)\}$ is a face of $\Omega_1(3,4,5)$, but it is not a facet of $\Omega_1(3,4,5)$, since the larger set $\{ (2,2), (2,3), (2,4), (3,2), (4,5)\}$ is also a face of $\Omega_1(3,4,5)$.
 \begin{figure}[H]
    \centering
     \begin{tikzpicture}[roundnode/.style={circle, draw=gray!80}, very thick, scale=0.85]
     \draw[very thin,gray!70]  (0.8,0.8) grid (4.2,5.2);
     \draw[black, very thick] (2,5)--(4,5)--(4,2);
     \draw[black, very thick] (1,5)--(1,4)--(2,4)--(2,1)--(4,1);
\node at (2,2){$\bullet$};
 \node at (2,3){$\bullet$};
 \node at (2,4){$\bullet$};
 \node at (3,2){$\bullet$};
 \node at (3,3){$\bullet$};
 \node at (3,4){$\bullet$};
\node[label=  left:{$(1,1)$}] at (1,1){$\bullet$};
 \node[label=  right:{$(4,5)$}] at (4,5) {$\bullet$};  
\node[label=  right:{$(4,1)$}] at (4,1) {}; 
\node[label=  right:{$(4,2)$}] at (4,2) {}; 
\node[label=  left:{$(1,5)$}] at (1,5) {}; 
\node[label=  above:{$(2,5)$}] at (2,5) {}; 
\end{tikzpicture}
%%%%%%%%%%%%%%%%%%%%%%%%%%%%%%%%%%%%%%%%%%%%seconddrawfin
\caption{The restriction to the set $V_1(3,4,5)$ of the family of nonintersecting paths $\mathcal F$ }
    \label{v1345}
\end{figure}
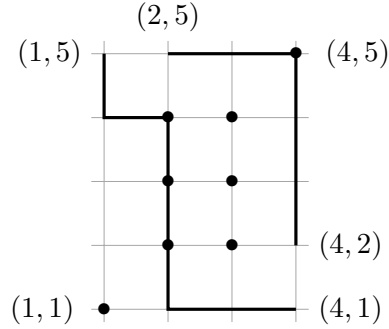
The facets of $\Omega_1(3,4,5)$ are $\{(1,1),  (2,2), (2,3), (2,4), (3,2)\}$, $\{(1,1), (2,4),\\
 (3,2), (3,3), (3,4)\}$, $\{(1,1), (2,3), (2,4), (3,2), (3,3)\}$, $\{(1,1),  (4,5)\}$, $\{(2,2), \\
 (2,3), (2,4),   (3,2),  (3,3), (3,4)\}$,  $\{ (2,2), (2,3), (2,4), (3,2), (4,5)\}$, $\{ (2,4), \\(3,2),  (3,3), (3,4), (4,5)\}$, and $\{ (2,3), (2,4), (3,2), (3,3), (4,5)\}$.
%\begin{align*}
%&\{(1,1),  (2,2), (2,3), (2,4), (3,2)\},  \{(1,1), (2,4), (3,2), (3,3), (3,4)\}, \\
%&\{(1,1), (2,3), (2,4), (3,2), (3,3)\},  \{(1,1),  (4,5)\},\\ 
%&\{(2,2), (2,3), (2,4),   (3,2),  (3,3), (3,4)\},  \{ (2,2), (2,3), (2,4), (3,2), (4,5)\},\\ 
%&\{ (2,4), (3,2), (3,3), (3,4), (4,5)\}, \text{ and } \{ (2,3), (2,4), (3,2), (3,3), (4,5)\}
%\end{align*}

In particular, it follows from the definitions that $\vert V_1(3,4,5)\vert=8$ and $\dim \Omega_1(3,4,5)= \max\{\vert F\vert -1 \mid F \text{ is a facet of } \Omega_1(3,4,5) \}=5$.
%\begin{align*}
%\dim \Omega_1(3,4,5)=&\max\{\dim \sigma\mid \sigma\in \Omega_1(3,4,5) \}\\
% =& \max\{\vert F\vert -1 \mid F \text{ is a facet of } \Omega_1(3,4,5) \}\\
% =&5.
%\end{align*}
\end{example}

\begin{lemma}\label{small}
For each $a,t,m,n$ under our assumptions,
the  cardinality of $V_a(t, m,n)$ is 
$$\vert V_a(t, m,n)\vert=h+t-1,$$ and the dimension of $\Omega_a(t, m,n)$ is 
\begin{align*}
 \dim \Omega_a(t, m,n)= \begin{cases}
			(t-1)(m+n-3t+3)-1, %h-(m-2t)(n-2t)-1, 
			& \text{if $l\ge t-1,$}\\
           h+t-l-2, &\text{if $l\le t-1,$}\\
		 \end{cases}
% \begin{cases}
%			(t-1)(m+n-3t+3)-1, 
%			& \text{if $\mini\{m,n\}\ge 2(t-1)$}\\
%           (\mini\{m,n\}-t+1)(\max\{m,n\}-t)+t-2, &\text{if $\mini\{m,n\}< 2(t-1)$}\\
%		 \end{cases}.
\end{align*}
where $h=(m-t+1)(n-t+1)$ is  the height of the ideal $\init(I_t)$ and $l=\mini\{m-t+1,n-t+1\}$. 
Moreover,
$$\beta_{h,h+t-1}(R/\init(I_t))\ge \sum_{0\le a\le t-1}\dim_K\Tilde{H}_{t-2}(\Omega_a(t, m,n);K).$$
\end{lemma}
\begin{proof} 
Since $ \dim R/I_t= (m+n-t+1)(t-1)$ (see, for example, \cite[Theorem 3.4.6]{BCRV}),
Macaulay’s theorem implies that the height of  $\init(I_t)$ is
\begin{align*}
h&=\dim R-\dim R/\init(I_t)=\dim R-\dim R/I_t\\
& =mn-(m+n-t+1)(t-1) \\
& =(m-t+1)(n-t+1).
\end{align*}
Hence $\vert V_a(t, m,n)\vert=(m-t+1)(n-t+1)+t-1=h+t-1$. In addition,
%The  cardinality of $ V_a(t, m,n)$ and the dimension of $\Omega_a(t, m,n)$ are obtained through direct application of the definitions. In particular,
\begin{align*}
 \dim \Omega_a(t, m,n)= &\max\{\dim \sigma\mid \sigma\in \Omega_a(t, m,n) \}\\
 =& \max\{\vert F\vert -1 \mid F \text{ is a facet of } \Omega_a(t, m,n) \}\\
 =& \begin{cases}
			h-(m-2t+2)(n-2t+2)-1, %h-(m-2t)(n-2t)-1, 
			& \text{if $l\ge t-1,$}\\
           h+t-1-l-1, 
           &\text{if $l\le t-1,$}\\
		 \end{cases}\\
 =& \begin{cases}
			(t-1)(m+n-3t+3)-1, %h-(m-2t)(n-2t)-1, 
			& \text{if $l\ge t-1,$}\\
           h+t-l-2, &\text{if $l\le t-1.$}\\
		 \end{cases}
% \begin{cases}
%			(t-1)(m+n-3t+3)-1, 
%			& \text{if $\mini\{m,n\}\ge 2(t-1)$}\\
%           (\mini\{m,n\}-t+1)(\max\{m,n\}-t)+t-2, &\text{if $\mini\{m,n\}< 2(t-1)$}\\
%		 \end{cases}.
\end{align*}
Moreover, using the result $\vert V_a(t, m,n)\vert=h+t-1$ and Hochster’s formula (\Cref{hoch}), 
\begin{align*}
\beta_{h,h+t-1}(R/\init(I_t))&=\beta_{h-1,h+t-1}(\init(I_t))\\
&=\beta_{h-1, h+t-1}(I_{\Delta(t,m,n)})\\
&=\sum_{\substack{\vert U\vert=h+t-1\\ U\se  \{1, \dots, m\}\times\{1,\dots, n\}}}\dim_K\Tilde{H}_{t-2}(\Delta(t,m,n)_U;K)\\
&\ge  \sum_{0\le a\le t-1}\dim_K\Tilde{H}_{t-2}(\Omega_a(t, m,n);K). \qedhere
\end{align*}
\end{proof}

%\subsection{Some properties of $\Omega_2(t, m,n)$}

\section{A study of the Betti number $\beta_{h, h+t-1}(R/\init(I_t))$}\label{sectionmain}
%\subsection{Preliminary}
This section is devoted to the proof of \Cref{maintheorem}. For the reader’s convenience, we recall the statement of  \Cref{maintheorem} below:
\begin{theorem1}
For all $t$ such that $2\le t \le \mini\{m,n\}$ and $t< \max\{m,n\}$,  the graded Betti number $$\beta_{h, h+t-1}(R/\init(I_t))\ge t,$$ where $h=(m-t+1)(n-t+1)$ is the height of $\init(I_t)$.
\end{theorem1}
As noted in \Cref{small}, the graded Betti number $\beta_{h,h+t-1}(R/\init(I_t))$ is bounded below by $\sum_{0\le a\le t-1} \dim_K\Tilde{H}_{t-2}(\Omega_a(t, m,n);K)$.  
%We therefore begin this section by studying 
Therefore,  to prove \Cref{maintheorem}, we study $\dim_K\Tilde{H}_{t-2}(\Omega_a(t, m,n);K)$
%study $\beta_{h,h+t-1}(R/\init(I_t))$, we show
%\begin{equation}\label{lem}
%\dim_K\Tilde{H}_{t-2}(\Omega_a(t, m,n);K)\ge 1
%\end{equation}
 by induction on $t$. More precisely,
\Cref{rmk2mn}, \Cref{rmk1}, \Cref{caset2} and \Cref{minmn2} address the base case $t=2$.  As a consequence of \Cref{caset2} and \Cref{minmn2}, \Cref{corocase2} provides the exact value of  the Betti number $\beta_{h, h+1}(R/\init(I_2))$. In  \Cref{mainlemma} we complete the argument  by showing that $\dim_K\Tilde{H}_{t-2}(\Omega_a(t, m,n);K)\ge 1$ holds for all $t,m,n,a$, thereby proving \Cref{maintheorem}.\\

First, observe from \Cref{w0233} that $\Omega_0(2,3,3)$ has two connected components. Hence
\begin{center} 
$\dim_K\Tilde{\H}_{0}(\Omega_0(2, 3,3);K)=$ the number of connected components of $ \Omega_0(2, 3,3)-1=2-1=1$.
\end{center} 
%, where  $\Tilde{H}_j(-)$ denotes the $j$-th reduced simplicial homology. 
More generally, for $t=2$, we have:
\begin{remark}\label{rmk2mn}
 For each $m,n\ge 2$ and for each $a\in \{0,1\}$,
$$\dim_K\Tilde{\H}_{0}(\Omega_a(2, m,n);K)= 1,$$
that is, $\Omega_a(2, m,n)$ has exactly two connected components.
\end{remark}

\begin{proof} 
Set 
     $f=\begin{cases}
			(m,n) , & \text{if $a=0$}\\
            (1,1) , & \text{if $a=1$}\\
		 \end{cases}$   
 and  $p=\begin{cases}
			(m-1,1), & \text{if $a=0$}\\
           (m,2), & \text{if $a=1$}\\
		 \end{cases}$   
(see \Cref{omega2mn}).
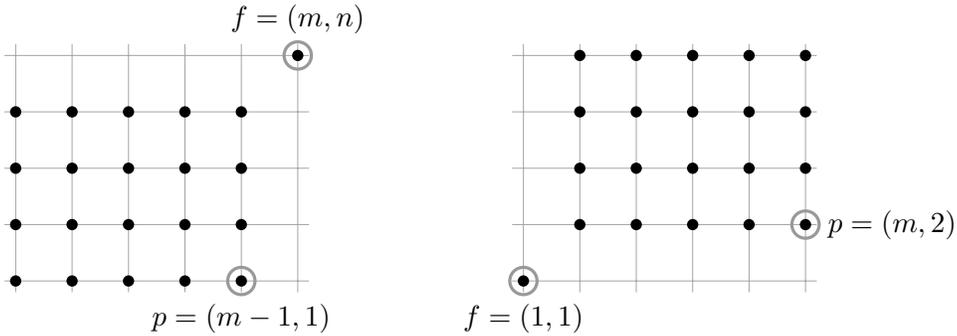
\begin{figure}[H]
    \centering
 \begin{tikzpicture}[roundnode/.style={circle, draw=gray!80}, very thick, scale=0.75] 
     \draw[very thin,gray!70]  (0.8,0.8) grid (6.2,5.2);

\draw[very thin,gray!70]  (9.8,0.8) grid (15.2,5.2);

 \foreach\x in{6}{
 \foreach\y in{5}{
   \fill (\x,\y)circle(0.1);
 }
 }
 \foreach\x in{10}{
 \foreach\y in{1}{
   \fill (\x,\y)circle(0.1);
 }
 }
  \foreach\x in{1,...,5}{
 \foreach\y in{1,...,4}{
 % \ifnum\x=\y, 
   \fill (\x,\y)circle(0.1);
%   \fi
 }
 }
  \foreach\x in{11,...,15}{
 \foreach\y in{2,...,5}{
 % \ifnum\x=\y, 
   \fill (\x,\y)circle(0.1);
%   \fi
 }
 }
 \node[roundnode]at  (6,5) {};
 \node[label=above:{$f=(m,n)$}] at (6,5) {};
  \node[roundnode]at  (5,1) {};
 \node[label=below:{$p=(m-1, 1)$}] at (5,1) {};
 \node[roundnode]at  (15,2) {};
 \node[label=right:{$p=(m,2)$}] at (15,2) {};
 \node[roundnode]at  (10,1) {};
 \node[label=below: {$f=(1,1)$}] at (10,1) {};

% \node[label=below:{$\V{\Omega_0(2, m,n)}$}] at (3.5,-0.4) {}; 
%  \node[label=below:{$\V{\Omega_1(2, m,n)}$}] at (12.5,-0.4) {}; 
 
\end{tikzpicture}
%%%%%%%%%%%%%%%%%%%%%%%%%%%%%%%%%%%%%%%%%%%%seconddrawfin
\caption{${V_0(2, m,n)}$ (left) and ${V_1(2, m,n)}$ (right)}
    \label{omega2mn}
\end{figure}
By \Cref{7441},  $F=\{f\}$ is a facet of $\Omega_a(2, m,n)$, and $q\le p$  for each $q\in V_a(2, m,n)\setminus F$. It follows that $$p\in \bigcap_{\substack{G\text{ is a facet of } \Omega_a(2, m,n)\\ G\not =F}} G.$$
Moreover,  $\langle F\rangle\cup\big(\Omega_a(2, m,n)\setminus F\big)=\Omega_a(2, m,n)$ and
$\langle F\rangle\cap\big(\Omega_a(2, m,n)\setminus F\big)=\{\emptyset\}$, where by $\langle F\rangle$ we mean the smallest simplicial complex containing $F$, and $\Omega_a(2, m,n)\setminus F=\{\sigma\in \Omega_a(2, m,n) \mid F\not\se \sigma\}$. Therefore, $\dim_K{\H}_{0}(\langle F\rangle;K)=1$,  $\dim_K{\H}_{0}(\Omega_a(2, m,n)\setminus F;K)= 1$ and
\begin{align*}
&\dim_K\Tilde{\H}_{0}(\Omega_a(2, m,n);K)\\
=&\dim_K{\H}_{0}(\Omega_a(2, m,n);K)-1 \\
=&\dim_K{\H}_{0}(\Omega_a(2, m,n)\setminus F;K)+ \dim_K{\H}_{0}(\langle F\rangle;K)-1 \\=&1 %=1. \qedhere
\end{align*}
by using Mayer–Vietoris sequence. 
\end{proof}

\begin{notation}
For each simplicial complex $\Delta$, we denote by $\V{\Delta}$ the vertex set of $\Delta$. For example, $\V{\Delta(t,m,n)}=\{1, \dots, m\}\times\{1,\dots, n\}$ and $\V{\Omega_a(t,m,n)}=V_a(t,m,n)$.
\end{notation}

\begin{remark}\label{rmk1}
Let $\Omega$ be a simplicial subcomplex of $\Delta(2,m,n)$  such that  $\dim_K\Tilde{\H}_{0}(\Omega;K)\ge 1$. Then 
\begin{enumerate}
\item[$i$)] $(1,n)\not\in \V{\Omega}$ and $(m,1)\not\in \V{\Omega}$. %, where $\V{\Omega}$ is the vertex set of $\Omega$
\item[$ii)$] For each $1\le j\le n$, there exists $1\le i\le m$ such that $(i,j)\not\in \V{\Omega}$; symmetrically, for each $1\le i\le m$, there exists $1\le j\le n$ such that $(i,j)\not\in \V{\Omega}$.
\end{enumerate}
If we assume furthermore that $\vert\V\Omega\vert=(m-1)(n-1)+1$, then
\begin{enumerate}
\item[$iii)$] for each $1\le j\le n$, there exists $1\le i\le m$ such that $(i,j)\in \V\Omega$; symmetrically, for each $1\le i\le m$, there exists $1\le j\le n$ such that $(i,j)\in \V\Omega$.
%\item[$iv)$]$\{(m-1, 1),(m,2)\}\not\se \V{\Omega}$.
\end{enumerate}
\end{remark}
%For simplicity, we introduce a notation: for each two 
\begin{proof}
\begin{enumerate}
\item[$i)$] Since $(1,n)\le (i,j)\le (m,1)$ for all $(i,j)\in  \V\Omega$,
if $(1,n)\in \V{\Omega}$ (resp. $(m,1)\in \V{\Omega}$), then $(1,n)$ (resp. $(m,1)$)  belongs to all facets of $\Omega$. %by  \cite[Proposition 4.4.1]{BCRV}. 
It follows that $\dim_K\Tilde{\H}_{0}(\Omega;K)=0$. \\
\item[$ii)$] If there exists $1\le l\le n$ such that  $(i,l)\in \V{\Omega}$ for all $1\le i\le m$, then  by \Cref{7441}, $G=\{(i,l)\mid 1\le i\le m\}$ is a face of $\Omega$, and for each facet $F$ of $\Omega$, there exists $1\le h\le m$ such that $(h, l)\in F$. It follows that $F\cap G\not=\emptyset$ for each facet $F$, that is, every facet $F$ is connected to $G$. Hence $\dim_K\Tilde{\H}_{0}(\Omega;K)=0$.\\
\item[$iii)$] Point $ii)$ implies that, for each $m,n \ge 2$, if $\dim_K\Tilde{\H}_{0}(\Omega;K)\ge 1$, then $mn \ge \vert\V\Omega \vert+ \max\{m,n\}.$\\
If there exists $1\le l\le n$ such that  $(i,l)\not\in \V{\Omega}$ for all $1\le i\le m$, then $\Omega$ can be considered as a simplicial subcomplex of $\Delta(2,m,n')$, where $n'=n-1$.  Hence 
\begin{align*}
m(n-1)=mn'&\ge \vert\V\Omega \vert+ n'\\
&=(m-1)(n-1)+1+n-1\\
&=m(n-1)+1,
\end{align*} a contradiction.
\end{enumerate}
\end{proof}

\begin{proposition}\label{caset2}
Assume $m,n\ge 3$. Let $\Omega$ be a simplicial subcomplex of $\Delta(2,m,n)$ such that $\vert\V\Omega\vert=(m-1)(n-1)+1$. Then $\dim_K\Tilde{\H}_{0}(\Omega;K)\ge 1$ if and only if $\Omega=\Omega_a(2, m,n)$ with $a=0,1$.
\end{proposition}

\begin{proof} 
\Cref{rmk2mn} shows that if $\Omega=\Omega_a(2, m,n)$ then  $\dim_K\Tilde{\H}_{0}(\Omega;K)= 1$. 
Now assume that $\dim_K\Tilde{\H}_{0}(\Omega;K)\ge 1$. The key idea in proving the other implication is  to  reduce the possible vertex region of $\Omega$ by repeatedly using this assumption. More precisely, \Cref{rmk1} $iii$) implies that $\{ 1\le i\le m \mid (i,n)\in \V{\Omega}\}\not=\emptyset$ and $\{ 1\le j\le n \mid (1,j)\in \V{\Omega}\}\not=\emptyset$. Set $$\h=\mini\{i\mid (i,n)\in \V{\Omega}\},$$
% and 
 $$f=\max\{j\mid (1,j)\in \V{\Omega}\}.$$  
It follows from \Cref{rmk1} $i$) that $2\le \h \le m$.
Applying again \Cref{rmk1} $iii$), we have $\{j\mid (i,j)\in\V{\Omega}\text{ with }  1\le i\le \h-1\}\not=\emptyset$. Set $$l=\max\{j\mid (i,j)\in \V{\Omega} \text{ with }  1\le i\le \h-1\}+1,$$ 
%and 
$$s=\mini\{i\mid (i,l-1)\in \V{\Omega}\}.$$
Therefore, $2\le l \le n$, $1\le f \le l-1$, $1\le s \le \h-1$, and $\V{\Omega}$ must be contained within the shaded region shown in \Cref{Vomega-1}.
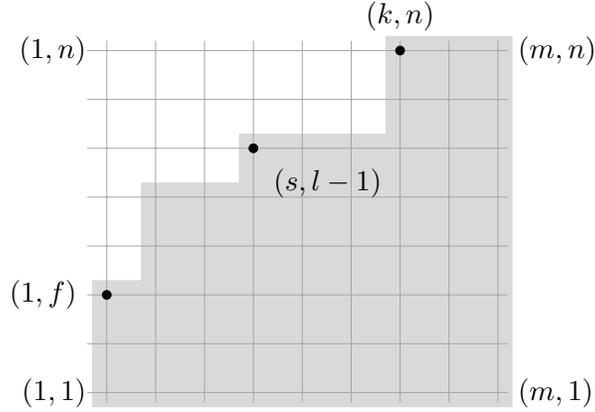
\begin{figure}[H]
    \centering
 \begin{tikzpicture}[scale=0.65]
 %[roundnode/.style={circle, draw=gray!80}, very thick, 

 \fill[fill=black!15, very thick] (0.7,0.7)--(0.7,3.3)--(1.7,3.3)--(1.7, 5.3)--(3.7, 5.3) --(3.7, 6.3) --(6.7, 6.3)--(6.7, 8.3) --(9.3, 8.3)--(9.3, 3.7)--(9.3, 0.7)--cycle;
 \draw[very thin,gray!70]  (0.6,0.8) grid (9.2,8.2);
  \foreach\x in{7}{
 \foreach\y in{8}{
   \fill (\x,\y)circle(0.1);
 }
 }
 \foreach\x in{4}{
 \foreach\y in{6}{
   \fill (\x,\y)circle(0.1);
 }
 }
 \foreach\x in{1}{
 \foreach\y in{3}{
   \fill (\x,\y)circle(0.1);
 }
 }
 
\node[label= left:{$(1,1)$}] at (1,1) {};
\node[label=  right:{$(m,n)$}] at (9,8) {};  
\node[label=  right:{$(m,1)$}] at (9,1) {}; 
\node[label= left:{$(1,n)$}] at (1,8) {}; 
\node[label=above:{$(\h,n)$}] at (7,8) {}; 
%\node[roundnode]at (7,8) {};
\node[label=below right:{$(s,l-1)$}] at (4,6) {}; 
%\node[roundnode]at (4,6) {};
\node[label=left:{$(1,f)\ $}] at (1,3) {}; 
 \end{tikzpicture}

%%%%%%%%%%%%%%%%%%%%%%%%%%%%%%%%%%%%%%%%%%%%%%%%%%%%second draw fin
\caption{Possible vertex region of $\Omega$.}
\label{Vomega-1}
\end{figure} 

We prove that if  $\dim_K\Tilde{\H}_{0}(\Omega;K)\ge 1$ then $\Omega=\Omega_a(2, m,n)$ for some $a\in
\{0,1\}$ through the following sequence of steps:
\begin{enumerate}
\item[Step 1:] Show that $\{(i,j)\mid \h\le i\le m, 1\le j\le f\}\cap \V{\Omega}=\emptyset.$
\item[Step 2:] Show that $s=1$, $f=l-1$,  and the vertices of $\Omega$ must be distributed as shown in \Cref{Vomega00}.
\item[Step 3:] Show that $(\h,l)\in\{(m,n), (2,2)\}$.  
\end{enumerate}
\begin{figure}[H]
    \centering
 \begin{tikzpicture}[roundnode/.style={circle, draw=gray!80}, very thick, scale=0.65]

 \draw[very thin,gray!70]  (0.6,0.8) grid (9.2,8.2);
  \foreach\x in{7,8,9}{
 \foreach\y in{5,6,7,8}{
   \fill (\x,\y)circle(0.1);
 }
 }
 \foreach\x in{1,2,3,4,5,6}{
 \foreach\y in{1,2,3,4}{
   \fill (\x,\y)circle(0.1);
 }
 }

\node[label=  left:{$(1,1)$}] at (1,1) {};
\node[label=  right:{$(m,n)$}] at (9,8) {};  
\node[label=  right:{$(m,1)$}] at (9,1) {}; 
\node[label=  left:{$(1,n)$}] at (1,8) {}; 
\node[label=above:{$(\h,n)$}] at (7,8) {}; 
%\node[roundnode]at (7,8) {};
%\node[label=below right:{$(s,l-1)$}] at (4,6) {}; 
%\node[roundnode]at (4,6) {};
\node[label=left:{$(1,l-1)\ $}] at (1,4) {}; 
%\node[roundnode]at (1,3) {};
\node[label=below:{$(\h,l)$}] at (6.1,6) {};
\node[label=below:{$(\h-1,1)$}] at (6.2,1.22) {};
\node[roundnode]at (7,5) {};
\node[roundnode]at (6,1) {};
%\draw[black!60, very thick] (0.8,0.8) rectangle (4.2,3.2);
%\draw[black!60, very thick] (4.4,3.7) rectangle (9.3,7.3);

%\draw[black, very thick] (6.7,0.7) rectangle (9.4,3.4);

 \end{tikzpicture}

%%%%%%%%%%%%%%%%%%%%%%%%%%%%%%%%%%%%%%%%%%%%%%%%%%%%second draw fin
\caption{}
\label{Vomega00}
\end{figure} 
Note that, given the conclusion of Step 3 is true:  If $(\h, l)=(m,n)$, then $\Omega=\Omega_0(2,m,n)$; if $(\h, l)=(2,2)$, then $\Omega=\Omega_1(2,m,n)$. This completes the proof.\\
What follows are the details of each step:\\
\noindent\textbf{Step 1:} Set $A=\{(i,j)\mid 1\le i\le m, 1\le j\le f\}$ and $B=\{(i,j)\mid \h\le i\le m, 1\le j\le n\}$. We want to show $A\cap B\cap \V\Omega=\emptyset$.\\
Since  $(1,f)\le (i,j)$ for all $(i,j)\in A$ and $(\h,n)\le (i,j)$ for all $(i,j)\in B$, if  there exists $(x,y)\in A\cap B\cap \V{\Omega}$, then $(1,f)\le (x,y)$, $(\h,n)\le (x,y)$, and $(i,j)\le (x,y)$ for all $(i,j)\in \V{\Omega}\setminus(A\cup B)$. (see \Cref{Vomega0})
\begin{figure}[H]
    \centering
 \begin{tikzpicture}%[scale=0.8]
 [roundnode/.style={circle, draw=gray!80}, very thick, scale=0.65]
  \fill[fill=black!15, very thick] (0.7,0.7)--(0.7,3.3)--(1.7,3.3)--(1.7, 5.3)--(3.7, 5.3) --(3.7, 6.3) --(6.7, 6.3)--(6.7, 8.3) --(9.3, 8.3)--(9.3, 3.7)--(9.3, 0.7)--cycle;
 %   \filldraw[black!7, very thick] (6.8,0.6) rectangle (9.2,8.3);
%\filldraw[black!7, very thick] (0.8,0.8) rectangle (9.4,3.2); 
 \draw[black, very thick] (6.8,0.6) rectangle (9.2,8.3);
\draw[black, very thick] (0.8,0.8) rectangle (9.4,3.2);
% \fill[fill=black!15, very thick] (0.7,0.7)--(0.7,3.3)--(1.7,3.3)--(1.7, 5.3)--(3.7, 5.3) --(3.7, 6.3) --(6.7, 6.3)--(6.7, 8.3) --(9.3, 8.3)--(9.3, 1.7)--(8.3, 1.7)--(8.3, 0.7)--cycle;
 \draw[very thin,gray!70]  (0.6,0.8) grid (9.2,8.2);
  \foreach\x in{7}{
 \foreach\y in{8}{
   \fill (\x,\y)circle(0.1);
 }
 }

 \foreach\x in{1}{
 \foreach\y in{3}{
   \fill (\x,\y)circle(0.1);
 }
 }
  \foreach\x in{8}{
 \foreach\y in{2}{
   \fill (\x,\y)circle(0.1);
 }
 }
   \foreach\x in{4}{
 \foreach\y in{6}{
   \fill (\x,\y)circle(0.1);
 }
 }
\node[label= left:{$(1,1)$}] at (0.8,1) {};
\node[label= right:{$(m,n)$}] at (9,8) {};  
\node[label= right:{$(m,1)$}] at (9.2,1) {}; 
\node[label=  left:{$(1,n)$}] at (1,8) {}; 
\node[label=above:{$(\h,n)$}] at (7,8) {}; 
\node[label=above:{$\quad(x,y)\qquad$}] at (8,2) {}; 
%\node[roundnode]at (7,8) {};
\node[label=below right:{$(s,l-1)$}] at (4,6) {}; 
%\node[roundnode]at (4,6) {};
\node[label=left:{$(1,f)\ $}] at (1,3) {}; 
\node[label=below:{$A$}] at (3,0.8) {}; 
\node[label=below:{$B$}] at (9.7,6) {}; 
 \end{tikzpicture}
\caption{}
\label{Vomega0}
\end{figure}
Therefore, the connections between vertices of $\Omega$ can be represented as shown in \Cref{od1}. 
%Therefore, the chains from $(1,n)$ to $(m,1)$ in the poset $(\V{\Omega},\le)$ can be represented as shown in \Cref{od1}. 
It follows that $\Omega$ has only one connected component, which contradicts  $\dim_K\Tilde{\H}_{0}(\Omega;K)\ge 1$.
\begin{figure}[H]
    \centering
 \begin{tikzpicture}%[scale=0.8]
 [roundnode/.style={circle, draw=gray!80}, very thick, scale=0.65]
	\draw[draw=black] (-1.00,2.00) circle (0.1);
    \node[label=above:{$(x,y)$}] at (-1,2) {};
     \draw[draw=black] (-5.00,0.00) circle (0.1);
    \node[label=above:{$(1,f)$}] at (-5,0) {};
	%\draw[draw=black] (-1.00,0.00) circle (0.1);
    \draw[draw=black] (3.00,0.00) circle (0.1);
    \node[label=above:{$(\h,n)$}] at (3,0) {};
	\draw[draw=black] (-2,0.00) circle (0.1);
    \node at (-1,0) {$\dots$};
	\draw[draw=black] (0.0,0.00) circle (0.1);
     \node[label=below:{$(i,j)\in \V{\Omega}\setminus(A\cup B)$}] at (-1,0) {};
	
    \node at (-5.00,-2.00) {$\dots$};
	\draw[draw=black] (-6.00,-2.00) circle (0.1);
	\draw[draw=black] (-4.00,-2.00) circle (0.1);
    \node[label=below:{$(i,j)\in A$}] at (-5,-2) {};
    
    \draw[draw=black] (2.00,-2.00) circle (0.1);
	\draw[draw=black] (4.00,-2.00) circle (0.1);
     \node at (3.00,-2.00) {$\dots$};
      \node[label=below:{$(i,j)\in B$}] at (3,-2) {};

	\draw[draw=black] (-1.00,2.00) -- (-5.00,0.00);
	\draw[draw=black] (-1.00,2.00) -- (-2,0.00);
	\draw[draw=black] (-1.00,2.00) -- (3.00,0.00);
	\draw[draw=black] (-1.00,2.00) -- (0,0.00);
    
	\draw[draw=black] (-5.00,0.00) -- (-6.00,-2.00);
	\draw[draw=black] (-5.00,0.00) -- (-4.00,-2.00);
    
	\draw[draw=black] (3.00,0.00) -- (2.00,-2.00);
	\draw[draw=black] (3.00,0.00) -- (4.00,-2.00);
\end{tikzpicture}
\caption{}
\label{od1}
\end{figure}
%By the definition of a facet of a simplicial subcomplex of  $\Delta(2,m,n)$ (recalled at the beginning of this section), $\Omega$ and the order complex of  $(\V{\Omega},\le)$  have the same number of connected components. Thus,  one can observe from \Cref{od1} that $\Omega$ has only one connected component, which contradicts  $\dim_K\Tilde{\H}_{0}(\Omega;K)\ge 1$. \\
We conclude this step by updating the figure of the possible vertex region  of $\Omega$ as the follows:
\begin{figure}[H]
    \centering
 \begin{tikzpicture}[scale=0.65]
 %[roundnode/.style={circle, draw=gray!80}, very thick, 

 \fill[fill=black!15, very thick] (0.7,0.7)--(0.7,3.3)--(1.7,3.3)--(1.7, 5.3)--(3.7, 5.3) --(3.7, 6.3) --(6.7, 6.3)--(6.7, 8.3) --(9.3, 8.3)--(9.3, 3.7)--(6.3, 3.7)--(6.3, 0.7)--cycle;
 \draw[very thin,gray!70]  (0.6,0.8) grid (9.2,8.2);
  \foreach\x in{7}{
 \foreach\y in{8}{
   \fill (\x,\y)circle(0.1);
 }
 }
 \foreach\x in{4}{
 \foreach\y in{6}{
   \fill (\x,\y)circle(0.1);
 }
 }
 \foreach\x in{1}{
 \foreach\y in{3}{
   \fill (\x,\y)circle(0.1);
 }
 }
 
\node[label=  left:{$(1,1)$}] at (1,1) {};
\node[label=  right:{$(m,n)$}] at (9,8) {};  
\node[label=  right:{$(m,1)$}] at (9,1) {}; 
\node[label=  left:{$(1,n)$}] at (1,8) {}; 
\node[label=above:{$(\h,n)$}] at (7,8) {}; 
%\node[roundnode]at (7,8) {};
\node[label=below right:{$(s,l-1)$}] at (4,6) {}; 
%\node[roundnode]at (4,6) {};
\node[label=left:{$(1,f)\ $}] at (1,3) {}; 
\end{tikzpicture}

%%%%%%%%%%%%%%%%%%%%%%%%%%%%%%%%%%%%%%%%%%%%%%%%%%%%second draw fin
\caption{}
\label{Vomega1}
\end{figure} 
  Set $V=\V{\Delta(2,m,n)}=\{1, \dots, m\}\times\{1,\dots, n\}$.  The shadow region in \Cref{Vomega1} can be written as $U=V\setminus (W_1\cup W_2\cup W_3\cup W_4)$, where 
$$W_1=\{(i,j)\mid 1\le i\le \h-1, l\le j\le n\},$$
$$W_2= \{(i,l-1)\mid 1\le i\le s-1\},$$
$$W_3=  \{(1,j)\mid f+1\le j\le l-2\},$$
$$W_4= A\cap B=\{(i,j)\mid \h\le i\le m, 1\le j\le f\}.$$
In particular, note that $\vert W_3\vert=l-2-f$ if and only if $ f\le l-2$.\\

\noindent\textbf{Step 2:} Our goal of this step is to prove $(s, l-1)=(1,f)$. Assume that $s=\mini\{i\mid (i,l-1)\in \V{\Omega}\}\ge 2$. It follows that $f\le l-2$. Therefore,   
$$\vert W_1\vert+\vert W_2\vert=(\h-1)(n-l+1)+(s-1)=(\h-1)(n-l)+\h+s-2,$$
and
$$\vert W_3\vert+\vert W_4\vert=(l-2-f)+f(m-\h+1)= f(m-\h)+l-2.$$
Since  $\V\Omega\se U$ and $W_i\cap W_j=\emptyset$ for each $i\not=j$,
\begin{align*}\label{areaU}
\vert V\vert-\vert\V\Omega\vert-\sum_{i=1}^4\vert W_i\vert
=&\vert V\vert-\vert\cup_{i=1}^4 W_i\vert-\vert\V\Omega\vert\\
=&\vert U\vert-\vert\V\Omega\vert\ge0,
\end{align*}
%\begin{center}
%   $ W=\{(i,j)\mid 1\le i\le h-1, l\le j\le n\}\cup \{(i,l-1)\mid 1\le i\le s-1\}\cup
%    \{(1,j)\mid f+1\le j\le l-2\}\cup
%    \{(i,j)\mid h\le i\le m, 1\le j\le f\}.$
%\end{center}
that is,
\begin{align*}
&mn-\Big((m-1)(n-1)+1\Big)-\sum_{i=1}^4\vert W_i\vert\\
=&m+n-2-(\h-1)(n-l)-\h-s+2- f(m-\h)-l+2\\
=&m-\h+n-l+(1-\h)(n-l)-s- f(m-\h)+2\\
=&(2-\h)(n-l)+(1-f)(m-\h)+(2-s)\ge 0.
\end{align*}
On the other hand,  our assumption $2\le s<\h\le m$ and the facts $f\ge 1$,  $n\ge l$ imply that $$(2-\h)(n-l)+(1-f)(m-\h)+(2-s)\le 0.$$
It follows that $\V{\Omega}=U$, $s=2$, $n=l$ and $i)$  $f=1$ or $ii)$ $m=\h$.
For both cases,  set $C=\{(i,j)\mid 1\le i\le \h-1, 1\le j\le n-1\}$ and $D=\{(i,j)\mid 2\le i\le m, f+1\le j\le n\}$ 
(see \Cref{Vomega3} for the case $f=1$).   Step 1 implies that $U\se C\cup D$, that is, $U=(U\cap C)\cup (U\cap D)$.
\begin{figure}[H]
    \centering
\begin{tikzpicture}
%[scale=0.8] 
[roundnode/.style={circle, draw=gray!80}, very thick, scale=0.7] 
% \fill[fill=black!15, very thick] (0.7,0.7) rectangle (6.4,7.3);  
  \draw[ very thick] (1.7,1.6) rectangle (9.3,8.3);
  \draw[very thick] (0.7,0.7) rectangle (6.4,7.3);
  \node[roundnode]at (1,1) {};
 \node[roundnode]at (6,1) {};
 \node[roundnode]at (9,2) {};
 \node[roundnode]at (2,7) {};
 \node[roundnode]at (7,8) {};
 \draw[very thin,gray]  (0.6,0.8) grid (9.2,8.2);
  \foreach\x in{7,...,9}{
 \foreach\y in{2,...,8}{
   \fill (\x,\y)circle(0.1);
 }
 }
 \foreach\x in{2,...,6}{
 \foreach\y in{2,...,7}{
   \fill (\x,\y)circle(0.1);
 }
 }
  \foreach\x in{1,...,6}{
 \foreach\y in{1}{
   \fill (\x,\y)circle(0.1);
 }
 }
 
\node[label=below left:{$(1,1)=(1,f)$}] at (1,1) {};
\node[label=below:{$\quad (\h-1,1)$}] at (6,1) {};
\node[label=right:{$(m,f+1)$}] at (9,2) {};
%\node[label=above right:{$(m,n)$}] at (9,8) {};  
%\node[label=below right:{$(m,1)$}] at (9,1) {}; 
%\node[label=above left:{$(1,n)$}] at (1,8) {}; 
\node[label=above:{$(\h,n)$}] at (7,8) {}; 
\node[label=above left:{$(s,l-1)=(2,n-1)$}] at (2,7) {}; 
\node[label=below:{$C$}] at (0.3,4) {}; 
\node[label=below:{$D$}] at (9.8,6) {}; 

 \end{tikzpicture}

%%%%%%%%%%%%%%%%%%%%%%%%%%%%%%%%%%%%%%%%%%%%%%%%%%%%second draw fin
\caption{$s=2$, $n=l$ and $f=1$}
\label{Vomega3}
\end{figure}
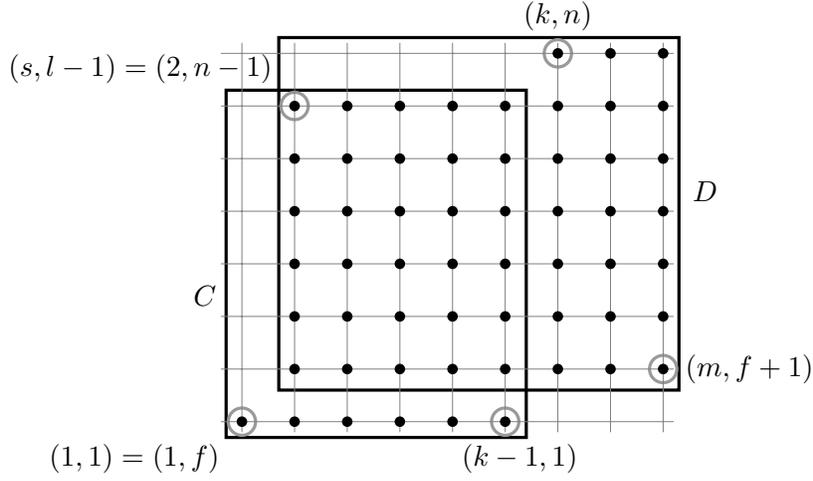 
\noindent In both cases,
 $(2,n-1)\in C\cap D\cap U$
holds by applying  the assumption $m,n\ge 3$ and using  $f\le l-2\le n-2$, $2=s\le \h-1<m$ again.
Moreover, $(i,j)\le(\h-1,1)$ for all $(i,j)\in C\cap U $ and $(i,j)\le (m,f+1)$ for all $(i,j)\in D\cap U$. In particular, $(2,n-1)\le (\h-1,1)$ and $(2,n-1)\le (m,f+1)$. 
It follows that,
as shown in \Cref{od2}, $\Omega$
has only one connected component. This leads to a contradiction.
\begin{figure}[H]
    \centering
 \begin{tikzpicture}%[scale=0.8]
 [roundnode/.style={circle, draw=gray!80}, very thick, scale=0.6]
	\draw[draw=black] (-1.00,2.00) circle (0.1);
    \node[label=above:{$(2, n-1)$}] at (-1,2) {};
     \draw[draw=black] (-5.00,0.00) circle (0.1);
    \node[label=above:{$(\h-1,1)$}] at (-5.2,0) {};
	%\draw[draw=black] (-1.00,0.00) circle (0.1);
    \draw[draw=black] (3.00,0.00) circle (0.1);
    \node[label=above:{$(m,f+1)$}] at (3.3,0) {};
	%\draw[draw=black] (-2,0.00) circle (0.1);
    %\node at (-1,0) {$\dots$};
	%\draw[draw=black] (0.0,0.00) circle (0.1);
  %   \node[label=below:{$(i,j)\in C\cap D \cap U$}] at (-1,0) {};
	
    \node at (-5.00,-2.00) {$\dots$};
	\draw[draw=black] (-6.00,-2.00) circle (0.1);
	\draw[draw=black] (-4.00,-2.00) circle (0.1);
    \node[label=below:{$(i,j)\in C\cap U$}] at (-5,-2) {};
    
    \draw[draw=black] (2.00,-2.00) circle (0.1);
	\draw[draw=black] (4.00,-2.00) circle (0.1);
     \node at (3.00,-2.00) {$\dots$};
      \node[label=below:{$(i,j)\in D\cap U$}] at (3,-2) {};

	\draw[draw=black] (-1.00,2.00) -- (-5.00,0.00);
	%\draw[draw=black] (-1.00,2.00) -- (-2,0.00);
	\draw[draw=black] (-1.00,2.00) -- (3.00,0.00);
	%\draw[draw=black] (-1.00,2.00) -- (0,0.00);
    
	\draw[draw=black] (-5.00,0.00) -- (-6.00,-2.00);
	\draw[draw=black] (-5.00,0.00) -- (-4.00,-2.00);
    
	\draw[draw=black] (3.00,0.00) -- (2.00,-2.00);
	\draw[draw=black] (3.00,0.00) -- (4.00,-2.00);

\end{tikzpicture}
\caption{}
\label{od2}
\end{figure} 
\noindent Therefore, $s=1$ and so $(1, l-1)\in \V{\Omega}$. Moreover, $f=\max\{j\mid (1,j)\in \V{\Omega}\}\le l-1$ implies $f=l-1$, that is, $(s, l-1)=(1,f)$. Thus,
$$W_2= \{(i,l-1)\mid 1\le i\le s-1\}=\emptyset,$$
$$W_3=  \{(1,j)\mid f+1\le j\le l-2\}=\emptyset,$$
and %Moreover, it follows that $W_2=W_3=\emptyset$. Hence 
the shadow region in \Cref{Vomega1} is  
\begin{align*}
    U=V\setminus(W_1\cup W_4)= \{(i,j)\mid i\le \h-1, j\le l-1\}\cup \{(i,j)\mid i\ge \h, j\ge l\},
\end{align*} which is the set of points shown in \Cref{Vomega00}.

\noindent\textbf{Step 3:} In this step we calculate $\h$ and $l$. By Step 2,
$$\vert U\vert=(\h-1)(l-1)+(m-\h+1)(n-l+1).$$
Now apply again the fact $\V\Omega\se U$:
\begin{align*}
0&\le (\h-1)(l-1)+(m-\h+1)(n-l+1)-(m-1)(n-1)- 1\\
&=mn-m(l-1)-n(\h-1)+2(\h-1)(l-1)-mn+m+n-2\\
&=m(2-l)+n(2-\h)+\underbrace{2(\h-1)(l-1)-2}_{=(\h-2)l+\h(l-2)}\\
%&=m(2-l)+n(2-\h)+2\hl-2\h-2l\\
%&=m(2-l)+n(2-\h)+(\h-2)l+\h(l-2)\\
&=(2-\h)(n-l)+(2-l)(m-\h).
\end{align*}
By applying again $m,n\ge 3$, $2\le \h\le m$ and $2\le l\le n$, we conclude that $U=\V{\Omega}$ %$\V{\Omega}$ must be distributed as shown in \Cref{Vomega00}.
and $(\h,l)\in\{(m,n), (2,2)\}$. 
\end{proof}

\begin{remark}\label{minmn2}
If $\mini\{m,n\}=2$, then
$$\sum_{\substack{\vert U\vert=(m-1)(n-1)+1\\ U\se \{1, \dots, m\}\times\{1,\dots, n\}}} \dim_K\Tilde{H}_{0}(\Delta(2,m,n)_U;K)=\max\{m,n\}-1.$$
\end{remark}
\begin{proof}
Note that
%the assumption $m,n\ge 3$  in \Cref{caset2}  is used only in Step 2 of the proof.
 if $\mini\{m,n\}=2$, then the result of Step 2 in \Cref{caset2} is trivial.  
In particular, assume $n=2<m$. For each $1\le b\le m-1$, consider the simplicial complex $\Omega_b$ defined as the restriction of   $\Delta(2,m,2)$  to the set
$$\{(1,1), \dots, (b,1), (b+1,2),\dots, (m,2)\}.$$ Since $\Omega_b$ has only two facets $F_1=\{(1,1), \dots, (b,1)\}$ and $F_2=\{ (b+1,2),\dots, (m,2)\}$, and since  $F_1\cap F_2=\emptyset$, $\Omega_b$ has exactly two connected components, that is, $\dim_K\Tilde{\H}_{0}(\Omega_b;K)= 1.$

If $\Omega$ is a simplicial subcomplex of $\Delta(2,m,2)$ such that $\vert\V\Omega\vert=(m-1)(n-1)+1=m$, by a discussion analogue to Step 1 and  Step 3 in \Cref{caset2},
$\dim_K\Tilde{\H}_{0}(\Omega;K)\ge 1$ if and only if
%(nothing that in the case $\mini\{m,n\}=2$, the result of Step 2 is trivial),  
$\Omega=\Omega_b$  with $1\le b\le m-1$, if and only if $\dim_K\Tilde{\H}_{0}(\Omega;K)= 1$.

The case $m=2$ is symmetric to the case $n=2$: If $\Omega$ is a simplicial subcomplex of $\Delta(2,2,n)$ such that $\vert\V\Omega\vert=n$, then  $\dim_K\Tilde{\H}_{0}(\Omega;K)\ge 1$ if and only if $\Omega=\Omega_b$ is the simplicial complex defined as the restriction of   $\Delta(2,2,n)$  to the set
$\{(1,1), \dots, (1,b), (2, b+1),\dots, (2,n)\}$
with $1\le b\le n-1$, if and only if $\dim_K\Tilde{\H}_{0}(\Omega;K)= 1$.
Therefore,  by Hochster's formula (\Cref{hoch}),
\begin{align*}
&\sum_{\substack{\vert U\vert=(m-1)(n-1)+1\\ U\se \{1, \dots, m\}\times\{1,\dots, n\}}} \dim_K\Tilde{H}_{0}(\Delta(2,m,n)_U;K)\\
&=\sum_{1\le b\le \max\{m,n\}-1}\dim_K\Tilde{H}_{0}(\Omega_b;K)\\
&=\max\{m,n\}-1. \qedhere
\end{align*}
\end{proof}

\begin{corollary}\label{corocase2}
%Let $\Omega$ be a simplicial subcomplex of $\Delta(2,m,n)$ such that $\vert\Omega\vert=(m-1)(n-1)+1$. If $\dim_K\Tilde{\H}_{0}(\Omega;K)\ge 1$, then $\dim_K\Tilde{\H}_{0}(\Omega;K)= 1$. In particular, 
For $m,n$ such that $2\le \mini\{m,n\}$ and $2<\max\{m,n\}$, 
\begin{equation*}
\beta_{h, h+1}(R/\init(I_2))=
\begin{cases}
2, &\text{ if } m,n\ge3,\\
\max\{m,n\}-1 , &\text{ if } \mini\{m,n\}=2,
 \end{cases}
\end{equation*}
where  $h$ is the height of  $\init(I_2)$.
% If $m,n\ge3$, then
%$$\beta_{h, h+1}(R/\init(I_2))=2,$$
%where  $h$ is the height of  $I_2$.
\end{corollary}
\begin{proof}
As discussed in the proof of \Cref{small}, for any $2\le t\le\mini\{m,n\}$, the height of the ideal $\init(I_t)$ is
$(m-t+1)(n-t+1).$
In particular, if $t=2$, then $h=(m-1)(n-1).$
%$$h=\dim R-\dim R/\init(I_2)=\dim R-\dim R/I_2=(m-1)(n-1).$$ 
Therefore, 
\begin{align*}
   \beta_{h, h+1}(R/\init(I_2)) &=\beta_{h-1, h+1}(I_{\Delta(2,m,n)})\\
   &=\sum_{\substack{U=h+1\\ U\se  \{1, \dots, m\}\times\{1,\dots, n\}}}\dim_K\Tilde{H}_{0}(\Delta(2,m,n)_U;K)\\
   &= \begin{dcases}
			\sum_{a\in\{0, 1\}}\dim_K\Tilde{H}_{0}(\Omega_a(2,m,n);K), & \text{if $m,n\ge 3$,}\\
       \max\{m,n\}-1, & \text{if $\mini\{m,n\}=2$,}
		 \end{dcases}\\
&=\begin{cases}
2, &\text{ if } m,n\ge3,\\
\max\{m,n\}-1 , &\text{ if } \mini\{m,n\}=2,
 \end{cases}
\end{align*}
 by reusing Hochster's formula (\Cref{hoch}) and applying \Cref{caset2}, \Cref{rmk2mn} and \Cref{minmn2}.\end{proof}
%\subsection{A proof of \Cref{maintheorem}}

 \Cref{maintheorem} follows directly from \Cref{small} together with the following result:
 
\begin{lemma}\label{mainlemma}
For each $0\le a\le t-1$, $$\dim_K\Tilde{\H}_{t-2}(\Omega_a(t,m,n);K)\ge 1.$$ 
\end{lemma}
\begin{proof}
We prove the statement by induction on $t\ge 2$. 
 \Cref{rmk2mn} shows that
$\dim_K\Tilde{\H}_{0}(\Omega_a(2,m,n);K)=1.$
Now assume that $t\ge 3$. For each $0\le a\le t-1$, set
\begin{align*}
    F=\begin{cases}
			\{(1,1)\}, & \text{if $a=t-1$,}\\
            \{(m,n)\}, & \text{if $0\le a\le t-2$.}
		 \end{cases}
\end{align*}
By definition, the link of $F$ in $\Omega=\Omega_a(t, m,n)$ is % (see, for example, \cite[Definition 5.3.4]{BH})
\begin{align*}
\lk_\Omega(F)&=\{G\mid F\cup G\in\Omega, F\cap G=\emptyset\}\\
&\iso \begin{cases}
			 \Omega_{t-2}(t-1, m-1,n-1), & \text{if $a=t-1$,}\\
            \Omega_a(t-1, m-1,n-1), & \text{if $0\le a\le t-2$}.
		 \end{cases}
\end{align*}
%Hence $\lk_\Omega(F)=  \Omega_b(t-1, m-1,n-1) $ for some $0\le b\le t-2$. 
Set 
$
b=
\begin{cases}
			{t-2}, & \text{if $a=t-1$,}\\
            a, & \text{if $0\le a\le t-2$}.
		 \end{cases}
$
By inductive hypothesis,
\begin{align*}
\dim_K\Tilde{\H}_{t-3}(\lk_\Omega(F);K) = \dim_K\Tilde{\H}_{(t-1)-2}(\Omega_b(t-1, m-1,n-1);K)\ge 1.
\end{align*}

Let $Y$ be a {\it geometric realization} (see, for example, \cite[Definition 5.2.8]{BH}) of the simplicial complex $\Omega=\Omega_a(t, m,n)$ given by $$\rho: V_a(t, m,n)\la \R^{d+1}$$
with $d=\dim \Omega$ as determined in \Cref{small},  and let $p\in\relint(\conv(\rho(F)))$ be an element of the relative interior of the convex hull of $\rho(F)$.  Therefore, $\Omega$ is a triangulation of the topological space $Y$, and a classical result in topology showed that the reduced singular homology of a topological space is isomorphic to the reduced simplicial homology of any of its triangulations (see \cite[Theorem 34.3]{Mun}, see also \cite[Theorem 5.3.2]{BH}). That is, 
\begin{align*}
\Tilde{\H}_{j}(Y; K)\iso  \Tilde{\H}_{j}(\Omega; K)
\end{align*}
and 
$$\Tilde{\H}_{j}(Y\setminus\{p\}; K)\iso  \Tilde{\H}_{j}(\Gamma; K)$$
for all $j$, where $\Gamma=\Omega\setminus F=\{\sigma\in \Omega \mid F\not\se \sigma\}$. 
Moreover, denote by $\H_{j}(Y, Y\setminus\{p\}; K)$ the $j$-th relative singular homology of the pair $(Y, Y\setminus\{p\})$. It follows directly from \cite[Lemma 5.4.5]{BH} that 
\begin{align*}
\H_{j}(Y, Y\setminus\{p\}; K)= \Tilde{\H}_{j-1}(\lk_\Omega(F);K)
\end{align*} for all $j$.
Furthermore, the following sequence
$$H_{t-2}(Y; K) \stackrel{\psi}\la \H_{t-2}(Y, Y\setminus\{p\}; K)\la H_{t-3}(Y\setminus\{p\}; K)\la H_{t-3}(Y; K)$$
is exact.
%Using \cite[Theorem 5.3.2]{BH}, we have $$\Tilde{\H}_{j}(Y; K))\iso  \Tilde{\H}_{j}(\Omega; K)$$ and
%$$\Tilde{\H}_{j}(Y\setminus\{p\}; K)\iso  \Tilde{\H}_{j}(\Gamma; K)$$
%for all $j$, where $\Gamma=\Omega\setminus F=\{\sigma\in \Omega \mid F\not\se \sigma\}$. 
%Moreover, 
%$$\H_{j}(Y, Y\setminus\{p\}; K)= \Tilde{\H}_{j-1}(\lk_\Omega(F);K)$$ for all $j$ by \cite[Lemma 5.4.5]{BH}. 
%Therefore,
%$$\Tilde{\H}_{t-2}(\Omega; K)\la \Tilde{\H}_{t-2}(\lk_\Omega(F);K)\la \Tilde{\H}_{t-3}(\Gamma; K)$$
%is exact.
Note that, since $t\ge 3$,   \Cref{7441}  implies that 
${\H}_{0}(\Omega;K)\iso K.$ Hence ${\H}_{0}(Y;K)\iso K$.
If $\Tilde{\H}_{t-3}(\Gamma; K)=0$, then 
\begin{align*}
H_{t-3}(Y\setminus\{p\}; K)&\iso {\H}_{t-3}(\Gamma; K)\iso\begin{cases}
			 0, & \text{if $t\ge 4$,}\\
          K\iso  {\H}_{t-3}(Y;K), & \text{if $t=3,$}
		 \end{cases}
\end{align*}
and so the natural map $\psi$ is surjective, from which it follows that
\begin{align*}
\dim_K \Tilde{\H}_{t-2}(\Omega; K) =&\dim_K {\H}_{t-2}(Y; K) \\
\ge& \dim_K \H_{t-2}(Y, Y\setminus\{p\}; K)\\
=&\dim_K \Tilde{\H}_{t-3}(\lk_\Omega(F);K)\\
 \ge &1.
\end{align*}
Therefore, it only remains to show $\Tilde{\H}_{t-3}(\Gamma; K)=0$. 

Denote by $\Gamma_{\le t-2}$ the $(t-2)$-skeleton of $\Gamma$, that is,
\begin{align*}
\Gamma_{\le t-2}=\{\sigma\in \Gamma\mid \dim\sigma\le t-2\}=\{\sigma\in \Gamma\mid \vert\sigma\vert\le t-1\}.
\end{align*}
\Cref{7441} implies that, for each $t-1$ vertices $v_1, \dots, v_{t-1}$ of $\Gamma$, the set $\{v_1, \dots, v_{t-1}\} $ is a face of $\Gamma$, which in turn implies that $\{v_1, \dots, v_{t-1}\}$ is a face of $ \Gamma_{\le t-2}$. Thus, $\Gamma_{\le t-2}$ can be considered as the $(t-2)$-skeleton of a simplex. 
We show that $\Gamma_{\le t-2}$ is Cohen–Macaulay through the following chain of implications: According to  \cite[Theorem 5.2.14]{BH}, every polytope is shellable; in particular, each simplex is shellable. \cite[Theorem 2.9]{BW} proved that the skeleta of shellable simplicial complexes remain shellable. Hence, $\Gamma_{\le t-2}$ is shellable. By \cite[Theorem 5.1.13]{BH}, shellable simplicial complexes are Cohen–Macaulay, and so $\Gamma_{\le t-2}$ is Cohen–Macaulay. 

A result of Reisner \cite[Theorem 1]{Rei} showed that a simplicial complex $\Sigma$ is Cohen–Macaulay if and only if  $\Tilde{\H}_{j}(\lk_\Sigma (\sigma); K)=0$ for each face $\sigma$ in $\Sigma$ and for each $j< \dim\lk_\Sigma (\sigma)$. Therefore, $\Tilde{\H}_{j}(\lk_{\Gamma_{\le t-2}} (\sigma); K)=0$ for each $\sigma\in\Gamma_{\le t-2}$ and for each $j< \dim\lk_{\Gamma_{\le t-2}} (\sigma)$. In particular, $$\Tilde{\H}_{j}(\Gamma_{\le t-2}; K)=\Tilde{\H}_{j}(\lk_{\Gamma_{\le t-2}} (\emptyset); K)=0$$ for each $j<\dim \Gamma_{\le t-2}=t-2.$ 
%Specifically, an $s$-simplex is the convex hull of $s+1$ affinely independent points (see \cite[Definition 5.2.5]{BH}). 
%By \cite[Theorem 5.2.14]{BH} and \cite[Theorem 2.9]{BW}, $\Gamma_{\le t-2}$ is a shellable simplicial complex, and so it is Cohen-Macaulay by \cite[Theorem 5.1.13]{BH}. 
%It follows that $\Tilde{\H}_{j}(\Gamma_{\le t-2}; K)=0$ for all $j<\dim \Gamma_{\le t-2}=t-2.$ Consequently, $\Tilde{\H}_{t-3}(\Gamma_{\le t-2}; K)=0$. %\cite[ex 5.3.15]{BH}.
Since  the module in homological degree $i$ of the \textit{augmented oriented chain complex} (see, for example,  \cite[Section 5.3]{BH}) of $\Gamma$ is
$$C_{i}(\Gamma)=\sum_{\substack{\sigma\in \Gamma\\ \dim\sigma\le i}}\Z\sigma=C_{i}(\Gamma_{\le t-2})$$
for $i\le t-2$,
using the definiton of reduced simplicial homology, we conclude
$\Tilde{\H}_{t-3}(\Gamma; K)=\Tilde{\H}_{t-3}(\Gamma_{\le t-2}; K)=0$. \end{proof}

\section{Lefschetz properties for $R/\init(I_t)$}\label{lef}
In this section, we study the weak and strong Lefschetz properties for $R/\init(I_t)$. 
The main results are \Cref{casem=n}, addressing the case $t=\mini\{m,n\}$, 
and \Cref{lasttheorem}, treating  the case  $t<\mini\{m,n\}$. 
%The case $t=\mini\{m,n\}$ is discussed in  \Cref{casem=n}. 
%The main result of this section is \Cref{lasttheorem}, which addresses the case  $t<\mini\{m,n\}$. % In particular, in \Cref{wiebi1} and \Cref{wiebe2} we show that similar statements to \cite[Proposition 2.8 and Theorem 3.1]{Wi} (which were originally proved for Artinian algebras) also hold for ideals that are not necessarily $\mm$-primary. \Cref{casem=n} follows as a consequence of \Cref{wiebe2}.
Applying \Cref{lasttheorem}, we present counterexamples in \Cref{answer}  that provide a negative answer to Murai’s question regarding the preservation  of Lefschetz properties under square-free Gr\"obner degenerations (\Cref{muquestion}). We conclude this section with \Cref{wiebe1} and \Cref{gin}, which includes a  brief discussion on the sharpness of the bound given in  \Cref{lasttheorem}. \\

Let $d$ be a positive integer. Recall that a homogeneous ideal $J$ of a standard graded polynomial ring $ S$ has a \textit{$d$-linear resolution} if $\beta_{i,i+j}(J) = 0$ for all $i$ and for all $j\not=d$.
%A homogeneous ideal $J$ is said to be \textit{componentwise linear} if $J_{\langle d\rangle}$ has a $d$-linear resolution for all $d\in\Z$. Here $J_{\langle d\rangle}$ denotes the ideal generated by all homogeneous polynomials of degree $d$ belonging to $J$. 
The following statement is an immediate consequence of a result of Conca and Varbaro  \cite[Corollary 2.7]{CV}.

\begin{remark}\label{linres}
Let $J$ be a homogeneous ideal of a standard graded polynomial ring $ S$ such that $\init_<(J)$ is a square-free monomial ideal for some monomial order $<$. Then $J$ has a {$d$-linear resolution} if and only if $\init_<(J)$ has a {$d$-linear resolution}.
\end{remark}
\begin{proof}
It follows directly from the definition that a homogeneous ideal $J$ of $S$ has a {$d$-linear resolution} if and only if $J$ is generated by homogeneous polynomials of degree $d$ and the Castelnuovo–Mumford regularity 
$$\reg(J)=\sup\{\beta_{i,i+j}(J)\not=0\}=d.$$
According to \cite[Corollary 2.7]{CV}, if the initial ideal $\init_<(J)$ is square-free, then the Castelnuovo–Mumford regularity of $J$ and of $\init_<(J)$ coincide, and thus the remark follows.
\end{proof}

\begin{proposition}\label{casem=n}
%Let $R=K[X]=K[X_{i,j}\mid 1\le i\le m, 1\le j\le n]$ be a standard graded polynomial ring and 
Let $t=\mini\{m,n\}$. Then $R/\init(I_t)$  has the SLP.
\end{proposition}
\begin{proof}
\cite[Theorem 1.1 (1)]{BC} showed that $I_t$ has a $t$-linear resolution. Since $\init(I_t)$ is square-free,  \Cref{linres} implies that $\init(I_t)$ has a $t$-linear resolution as well.
If $\underline\theta=\theta_1, \dots, \theta_{\dim R/\init(I_t)}\in R_1$ is a linear system of parameters, then there exists a homogeneous ideal $J$ of the polynomial ring $K[Y_1, \dots, Y_h]$ such that $J$ has a $t$-linear resolution and
\begin{align*}
R/(\init(I_t), \underline\theta)\iso K[Y_1, \dots, Y_h]/J,
\end{align*}
where $h=\max\{m,n\}-t+1$ is the height of $\init(I_t)$. Moreover,  since $R/\init(I_t)$ is 
Cohen–Macaulay  (see \cite[Theorem 4.4.5]{BCRV}), 
if follows from \cite[Theorem 3.2]{RV} (see also \cite[Proposition 2.1]{CRV}) that 
$J$ has a $t$-linear resolution if and only if
$$\HF(J, t)=\binom{h+t-1}{t}.$$
Therefore, $J=(Y_1, \dots, Y_h)^t$. It is easy to verify that, for a general linear form $L\in (Y_1, \dots, Y_h)$, the multiplication map 
$$\times L^s: \big(K[Y_1, \dots, Y_h]/(Y_1, \dots, Y_h)^t\big)_j\la \big(K[Y_1, \dots, Y_h]/(Y_1, \dots, Y_h)^t\big)_{j+s}$$ is surjective for each $s\ge 1$ and for each $j \ge t-s$, and injective for each $j \le t-s-1$. Hence 
 $K[Y_1, \dots, Y_h]/J$ has the SLP, and it follows that $R/\init(I_t)$ has the SLP.
\end{proof}

\begin{theorem}\label{lasttheorem}
Let  $2\le t < \mini\{m,n\}$. Then
$R/\init(I_t)$ fails the WLP if one of the following conditions holds:
\begin{itemize}
\item[$i$)] $t=2$ and $mn\ge 16$,
\item[$ii$)] $t=3$ and $mn\ge 24$,
\item[$iii$)] $t\ge4$ and $mn\ge (t+1)(t+2)$.
%\item[$iv$)] $t\ge5$ and $  mn\ge (t+1)(t+2)$.
\end{itemize}
In particular, if $m=n$, then $R/\init(I_t)$ fails the WLP for all $n\ge t+2.$ 
\end{theorem}

The following  two statements serve as preparation for the proof of  \Cref{lasttheorem}.

\begin{lemma}\label{rmkapp2}
Let  $2\le t < \mini\{m,n\}$.
Set $$F_t(m,n)=\binom{h+t-2}{t}-\binom{m}{t}\binom{n}{t},$$ where $h=(m-t+1)(n-t+1)$ is the height of the ideal $\init(I_t)$.   If $F_t(m,n)\ge 0$, then $R/\init(I_t)$ fails the WLP.
\end{lemma}
\begin{proof}
Let $\underline\theta=\theta_1,\dots, \theta_{mn-h}\in R_1$ be a sequence of general linear forms.  
%We want to show that there exists $j$ such that
%the multiplication map
%$\times L: \Big[R/(\init(I_t), \underline\theta)\Big]_{j}\la \Big[R/(\init(I_t), \underline\theta)\Big]_{j+1}$
%is neither injective nor surjective for any linear form $L$. 
%By \Cref{maintheorem} and  \Cref{rmkapp}, when $j=t-1$, 
\Cref{maintheorem} and  \Cref{rmkapp} imply that
the multiplication map 
$$\times L: \Big[R/(\init(I_t), \underline\theta)\Big]_{t-1}\la \Big[R/(\init(I_t), \underline\theta)\Big]_{t}$$
fails  to be injective  for any linear form $L$. Therefore, by \Cref{corormkapp}, to prove \Cref{rmkapp2}, it suffices to show that the above map $\times L$ also fails to be surjective  for any linear form $L$.

Using again the fact that the $t$-minors of  $X$  form a Gr\"obner basis of $I_t$ with respect to a diagonal monomial order, we have $$\mu=\beta_0(\init(I_t))=\beta_0(I_t)=\binom{m}{t}\binom{n}{t}$$ is the minimal number of generators of $\init(I_t)$.
Since $\underline\theta$ % with $d=\dim R/\init(I_t)=(m+n-t+1)(t-1)$ 
is a linear system of parameters, there exists a homogeneous ideal $J=(f_1, \dots, f_\mu)$ of the polynomial ring $S=K[Y_1, \dots, Y_h]$ such that $R/(\init(I_t),\underline\theta)\iso S/J$, where  $f_i\in S$ are homogeneous with $\deg (f_i)=t$ for all $i$.
%(see, for example, \cite[Theorem 5.1.16 (b)]{BH}). 
%By \cite[Theorem 4.3.2(a)]{BCRV}, $$\mu=\beta_0(I_t)=\binom{m}{t}\binom{n}{t}.$$
Therefore, the Hilbert function of $R/(\init(I_t), \underline\theta)$ in $t-1$ and in $t$ can be calculated as follows: 
\begin{align*}
    \HF(R/(\init(I_t), \underline\theta), t-1)=&\HF(S/J, t-1)=\binom{h+(t-1)-1}{h-1}\\
    =&\binom{h+t-2}{h-1}
\end{align*}
is the number of monomials of degree $t-1$ belonging to $S$, and
\begin{align*}
   & \HF(R/(\init(I_t), \underline\theta), t)=\HF(S/J, t)=\binom{h+t-1}{h-1}-\mu\\
    &=\binom{h+t-2}{h-1}+\binom{h+t-2}{h-2}-\binom{m}{t}\binom{n}{t}\\
    &=\HF(R/(\init(I_t), \underline\theta), t-1)+\binom{h+t-2}{t}-\binom{m}{t}\binom{n}{t}\\
&=\HF(R/(\init(I_t), \underline\theta), t-1)+F_t(m,n)
\end{align*}
is the number of monomials of degree $t$ belonging to the monomial basis of $S/J$. 
Consequently, if $F_t(m,n)>0$, then
$\dim_K\Big[R/(\init(I_t), \underline\theta)\Big]_{t-1}<\dim_K\Big[R/(\init(I_t), \underline\theta)\Big]_{t}$, which implies that
 the map
$$\times L: \Big[R/(\init(I_t), \underline\theta)\Big]_{t-1}\la \Big[R/(\init(I_t), \underline\theta)\Big]_{t}$$
fails to be surjective for any linear form $L\in R_1$. 
If  $F_t(m,n)= 0$, then  $\dim_K\Big[R/(\init(I_t), \underline\theta)\Big]_{t-1}=\dim_K\Big[R/(\init(I_t), \underline\theta)\Big]_{t}$,  and therefore $\times L$ 
necessarily fails to be surjective because it is not injective. Thus, if $F_t(m,n)\ge 0$,  then $R/\init(I_t)$ fails the WLP.
\end{proof}

\begin{remark}\label{gin456}
If $t=4$ and $(m,n)\in\{(5,6), (6,5)\}$, then  $R/\init(I_4)$ fails the WLP;
if $t=3$ and $(m,n)\in\{(4,5), (5,4)\}$,  then  $R/\init(I_3)$  has the WLP but fails the SLP.
\end{remark}

\begin{proof}
Assume that $t=4$, $m=5$, and $n=6$.
Since the intersection of two nonempty Zariski open sets is nonempty, it follows from \Cref{openset} that if there exists a Zariski open set $U$ such that for any sequence of linear forms $\theta_1,\dots,\theta_{24}, L \in U$, the linear form $L$ is not a Lefschetz element for $R/(\init(I_4),\underline\theta)$, then $R/\init(I_4)$ fails the WLP. Moreover, by \Cref{7441} and \Cref{SRring}, for each $t,m,n\in\N$  such that $2\le t\le \mini\{m,n\}$ and $t<\max\{m,n\}$, there are $t(t-1)$ variables of $R$ that do not appear in $\init(I_t)$. In the present case, $t(t-1)=12$. To simplify notation, we relabel the $12$ variables that do not appear in $\init(I_4)$ as $z_1, \dots, z_{12}$, and the remaining variables as $y_1, \dots, y_{18}$, ordered so that
$y_1 > y_2 > \dots > y_{18}$,
in such a way that the relabeling is compatible with the given ordering
$X_{1,1} > X_{1,2} > \dots > X_{m,n}$.
That is, whenever $X_{a,b}$ and $X_{c,d}$ are relabeled as $y_i$ and $y_j$, respectively, and $X_{a,b} > X_{c,d}$, then $i < j$. 
Hence, we may take $U$ to be the Zariski open subset determined by $z_1,\dots, z_{12}$ together with sufficiently general linear forms $\delta_1, \dots, \delta_{12}, l$ in the variables $y_1,\dots,y_{18}$. 

Set   $A=K[y_1,\dots, y_{18}]$.
\Cref{ginconca} implies that the Hilbert function of $A/(\init(I_4), \underline \delta, l)$ is equal to the
Hilbert function of $A/(\gin(\init(I_4)), y_6,\dots, y_{18})$, and  the Hilbert function of  $A/(\init(I_4), \underline \delta)$ is equal to the
Hilbert function of $A/(\gin(\init(I_4)), y_7,\dots, y_{18})$.
%, and $\gin(\init(I_4))$ denotes the generic initial ideal of $init(I_4)$ with respect to the reverse lexicographic order. 
Furthermore, a Macaulay2 \cite{M2} \cite{GenericInitialIdealSource} computation showed that the degree $4$ component of $\gin(\init(I_4))$ %\footnote{A complete description of the ideal $\gin(\init(I_4))$, including all generators, can be found in \url{https://github.com/hoyu26/gininI456}.} 
is
\begin{align*}
\gin(\init(I_4))_4=(y_1, y_2,y_3, y_4)^4+(y_1, y_2,y_3, y_4)^3y_5+
(y_1, y_2,y_3, y_4)^2y_5^2+\\
(y_1y_5^3, y_2y_5^3, y_3y_5^3, y_1^3y_6, y_1^2y_2y_6, y_1y_2^2y_6, y_2^3y_6, y_1^2y_3y_6, y_1y_2y_3y_6, y_2^2y_3y_6).
\end{align*}
Therefore,
\begin{align*}
 \HF(R/(\init(I_4), \underline z, \underline \delta, l), 4)= &\HF(A/(\init(I_4), \underline \delta, l), 4)\\
 = &\HF(A/(\gin(\init(I_4)), y_6,\dots, y_{18}), 4)\\
 =&2\not=0.
\end{align*}
While
\begin{align*}
\HF(R/(\init(I_4), \underline z, \underline \delta), 4)=\HF(A/(\gin(\init(I_4)), y_7,\dots, y_{18}), 4)=51,
\end{align*}
and
\begin{align*}
\HF(R/(\init(I_4), \underline z, \underline \delta), 3)=\HF(K[y_1,\dots, y_{6}], 3)=56>51.
\end{align*}
It follows  that the multiplication map $$\times l: \big[R/(\init(I_4), \underline z, \underline \delta)\big]_3\la \big[R/(\init(I_4), \underline z, \underline \delta)\big]_4$$ fails to have maximal rank. \footnote{In the proof above, one may equivalently consider general linear forms in $R_1$ or choose a different Zariski open subset. The specific choice of the open set $U$ is made only to simplify the computation of the generic initial ideal. 
Moreover, a complete description of the ideal $\gin(\init(I_4))$, including all its generators, is available at \url{https://github.com/hoyu26/gininI456}.}

Similarly, for the case $t = 3$, $m = 4$, and $n = 5$, we relabel the variables
$X_{1,1}, \dots, X_{4,5}$
as
$z_1, \dots, z_6, \; y_1, \dots, y_{14}$,
where $z_1, \dots, z_6$ are the variables that do not appear in $\init(I_3)$, and the ordering on the $y_i$ is taken to be compatible with the ordering on the $X_{a,b}$.
The generic initial ideal of $\init(I_3)$ with respect to the reverse lexicographic order is 
\begin{align*}
\gin(\init(I_3))=(y_1, y_2,y_3, y_4,y_5)^3+y_6(y_1^2, y_1y_2,y_2^2, y_1y_3, y_2y_3)+\\
y_6^2(y_3^2, y_1y_4, y_2y_4, y_2y_4, y_4^2, y_1y_5, y_2y_5, y_3y_5)+y_6^3(y_1, y_2)+\\
(y_4y_5y_6^3, y_5^2y_6^3, y_3y_6^4, y_4y_6^4, y_5y_6^4,  y_6^5).
\end{align*}
In particular, 
\begin{align*}
\gin(\init(I_3))_4=(y_1, y_2,y_3, y_4,y_5)^4 +(y_1, y_2,y_3, y_4,y_5)^3y_6+\\
(y_1, y_2, y_3)(y_1, y_2,y_3, y_4,y_5)y_6^2+y_4^2y_6^2+(y_1, y_2)y_6^3
+J
\end{align*}
with $J\se (y_7,\dots, y_{14})$. 
Set $A'=K[y_1,\dots, y_{14}]$. If $\delta'_1, \dots, \delta'_{8}, l'\in A'_1$ are general linear forms, applying \Cref{ginconca} once again, it follows that  the Hilbert function of  $A'/(\init(I_3), \underline \delta')$ is equal to the
Hilbert function of $A'/(\gin(\init(I_3)), y_7,\dots, y_{14})$ which is $(1,6,21,16,6,0,\dots)$. Moreover,
\begin{align*}
\HF(A'/(\init(I_3), \underline \delta', l'^3), 4)
 = \HF(A'/(\gin(\init(I_3)), y_6^3, y_7,\dots, y_{14}), 4)=2.
\end{align*}
Therefore, the multiplication map $$\times l'^3: \big[R/(\init(I_3), z_1, \dots, z_6, \underline \delta')\big]_1\la \big[R/(\init(I_3), z_1, \dots, z_6, \underline \delta')\big]_4$$ fails to have maximal rank. Furthermore, since
\begin{align*}
\HF(A'/(\init(I_3), \underline \delta', l'), j)
 = \HF(A'/(\gin(\init(I_3)), y_6, y_7,\dots, y_{14}), j)\\
 =\HF(K[y_1,\dots, y_5]/(y_1, y_2,y_3, y_4,y_5)^3,j) \\
= \max\{\HF(A'/(\init(I_3), \underline \delta', l'), j)-\HF(A'/(\init(I_3), \underline \delta', l'), j-1),0\}
\end{align*}
 for all $j$, it follows that $R/\init(I_3)$ has the WLP.
  \end{proof}

\begin{proof}[Proof of \Cref{lasttheorem}]
By \Cref{rmkapp2}, a sufficient condition for ``$R/\init(I_t)$ fails the WLP'' is that $$F_t(m,n)=\binom{h+t-2}{t}-\binom{m}{t}\binom{n}{t} \ge 0.$$
Thus, to understand when $R/\init(I_t)$ fails the WLP, we first study when 
%to conclude that $R/\init(I_t)$ fails the WLP, we need to investigate when
$F_t(m,n)\ge 0$.\\
Note that if $m=n=t+1$, then 
\begin{align*}
F_t(t+1,t+1)=&\frac{(t+1)(t+2)}{2}-(t+1)^2\\
=&-\frac{t(t+1)}{2}
<0
\end{align*}
 for all $t\ge 2$. Hence we may assume that $(m-t)(n-t)\ge2$.\\
For each $0\le j\le t-2$,  
\begin{align*}
&(h+j)(j+2)-(m-t+2+j)(n-t+2+j)\\
=&(j+2)\Big(h+j-(m+n-2t)-(j+2)\Big)-(m-t)(n-t)\\
=&(j+2)\Big((m-t)(n-t)-1\Big)-(m-t)(n-t)\\
\ge&2\Big((m-t)(n-t)-1\Big)-(m-t)(n-t)\\
=&(m-t)(n-t)-2\\
\ge& 0,
\end{align*}
and so
$$\prod_{j=0}^{t-2}(h+j)\ge \prod_{j=0}^{t-2}\frac{(m-t+2+j)(n-t+2+j)}{j+2}.$$
It follows that
$$\frac{n-t+1}{(t-1)!}\prod_{j=0}^{t-2}(h+j)\ge\frac{n-t+1}{(t-1)!} \prod_{j=0}^{t-2}\frac{(m-t+2+j)(n-t+2+j)}{j+2},$$
that is,
\begin{align}\label{last}
(n-t+1)\binom{h+t-2}{t-1}\ge\binom{m}{t-1}\binom{n}{t}.
\end{align}
Recall the well-known Vandermonde identity, which states that
\begin{align*}\label{Vi}
    \binom{u+v}{w} = \sum_{j = 0}^w\binom{u}{j} \binom{v}{w-j}
    \end{align*} for all $u,v,w\in\N.$
Therefore, by using the Vandermonde identity and applying Inequality (\ref{last}),
\begin{align*}
&F_t(m+1,n)-F_t(m,n)\\
=&\underbrace{\binom{h+(n-t+1)+t-2}{t}}_{=\sum_{j=0}^{t}\binom{n-t+1}{j}\binom{h+t-2}{t-j}}-\underbrace{\binom{m+1}{t}\binom{n}{t}}_{=\binom{m}{t}\binom{n}{t}+\binom{m}{t-1}\binom{n}{t}}-\binom{h+t-2}{t}+\binom{m}{t}\binom{n}{t}\\
=&\sum_{j=1}^{t}\binom{n-t+1}{j}\binom{h+t-2}{t-j}-\binom{m}{t-1}\binom{n}{t}\\
=&(n-t+1)\binom{h+t-2}{t-1}-\binom{m}{t-1}\binom{n}{t}+\sum_{j=2}^{t}\binom{n-t+1}{j}\binom{h+t-2}{t-j}\\
\ge&\sum_{j=2}^{t}\binom{n-t+1}{j}\binom{h+t-2}{t-j}\\\ge &\binom{n-t+1}{2}>0. 
\end{align*}
%The first inequality above follows from .\\
Symmetrically, we also have $F_t(m,n+1)>F_t(m,n)$ for all $m,n$ such that $(m-t)(n-t)\ge2$. 
Since
\begin{align*}
F_t(t+1,t+2)=F_t(t+2,t+1)&=\binom{4+t}{t}-\binom{t+1}{t}\binom{t+2}{t} \notag\\
&=\frac{(t+1)(t+2)}{24}t(t-5), 
\end{align*}
it follows that 
\begin{itemize}
\item $F_t(m,n)\ge 0$ for all $t\ge 5$ and for all $mn\ge (t+1)(t+2)$. 
\end{itemize}
%$F_t(m,n)\ge F_t(t+1,t+2)\ge0$ for all $t\ge 5$ and for all $mn\ge (t+1)(t+2)$. 
Additionally, some direct calculations show that $F_2(3,6)=F_2(6,3)=0$,  $F_t(t+1,t+3)=F_t(t+3,t+1)> 0$ for $t=3,4$, and $F_t(t+2,t+2) \ge 0$  for $2\le t\le 4$. Thus, 
\begin{itemize}
\item $F_2(m,n)\ge 0$ for all $mn\ge 16$,
\item $F_3(m,n)\ge 0$ for  all $mn\ge 24$,
\item $F_4(m,n)\ge 0$ for all $mn\ge 35$.
\end{itemize}
%$F_2(m,n)\ge 0$ for all $mn\ge 16$, $F_3(m,n)\ge 0$ for  all $mn\ge 24$ and $F_4(m,n)\ge 0$ for all $mn\ge 35$. 
%Short Exact Sequence \ref{ses} implies that $R/\init(I_t)$ has the WLP if and only if  
%\thickmuskip=0mu
%$$\HF(R/(\init(I_t), \underline\theta,L),j)=\max\{0, \HF(R/(\init(I_t), \underline\theta),j)-\HF(R/(\init(I_t), \underline\theta),j-1)\}$$ 
%\thickmuskip=5mu plus 5mu
%for all $j$, where $\underline\theta=\theta_1, \dots, \theta_{\dim R/I_t} \in R_1$ is a sequence of general linear forms,  and $L\in R_1$ is a general linear form.
Moreover,  if $t=4$ and $(m,n)\in\{(5,6), (6,5)\}$, then  $R/\init(I_4)$ fails the WLP by \Cref{gin456}.
%the Hilbert function of $R/(\init(I_t), \underline\theta)$ is  $$(1,6,21,56,51,30,10)$$ and  the Hilbert function of $R/(\init(I_t), \underline\theta,L)$ is $$(1,5,15,35,2,0,0).$$ Therefore,  
%\begin{align*}
%&\HF(R/(\init(I_t), \underline\theta,L),4)=2\\
%\not=&\max\{0,\HF(R/(\init(I_t), \underline\theta),4) - \HF(R/(\init(I_t), \underline\theta),3)\}
%\end{align*} It follows that
% $R/\init(I_4)$ with $(m,n)\in\{(5,6), (6,5)\}$ fails the WLP.
Therefore, if  one of the following conditions holds:
\begin{itemize}
\item[$i$)] $t=2$ and $mn\ge 16$,
\item[$ii$)] $t=3$ and $mn\ge 24$,
\item[$iii$)] $t\ge4$ and $mn\ge (t+1)(t+2)$,
\end{itemize}
then $R/\init(I_t)$ fails the WLP.
\end{proof}

\begin{remark}\label{answer}
Let us focus on the case $m=n$. In their paper \cite{SW}, Soll and Welker defined a monomial order $\prec$ (\cite[Definition 28]{SW}) on $R=K[X]$, and they conjectured that the simplicial complex $\Delta$ defined by $\init_{\prec}(I_t)$ is a simplicial $(d-1)$-sphere 
for all $t \le n$ (\cite[Conjecture 13]{SW}, see also \cite[Conjecture 17]{SW}). %where $\prec$ is the monomial order defined in \cite[Definition 28]{SW}. 
They proved this conjecture for the cases $t=2$ and $t=n-1$.
In \cite[Theorem 7.1]{RS} \footnote{However, their proof was not formally published and is only available in the arXiv version of their paper.}, Rubey and Stump  provided a complete proof of this conjecture for all $t$.

Moreover, in proving and applying McMullen’s g-conjecture, which states that $g$-vector of a simplicial sphere is the $f$-vector of a multicomplex  \cite{Mc} (for the proof see \cite{St} \cite{BL}, see also \cite{APP} \cite{PP}), Adiprasito  showed that the corresponding Stanley–Reisner ring  of a simplicial $(d-1)$-sphere has the SLP \cite[Theorem I]{Ad}. Consequently,  $R/\init_{\prec}(I_t)$ has the SLP for all $t \le n$.  It follows from \Cref{WiebeMurai} that $R/I_t$ has the SLP for all $t \le n$. 
%According to \cite[Lemma 3.3]{Mu}, which we recalled in \Cref{intro}, the Lefschetz properties transfer from the initial ideal to the ideal. Thus, it follows that $R/I_t$ has the SLP for all $t \le n$. 
Therefore, for $n\ge t+2$, \Cref{lasttheorem} provides a family of ideals $I_t$ such that  $R/I_t$ has the SLP and $\init(I_t)$ is square-free, but  $R/\init(I_t)$ fails  the WLP.
\end{remark}

\begin{examples}
%By running  the Macaulay2 code from \Cref{M2code},
We further present the following counterexamples, which follow directly from \Cref{gin456} and \Cref{lasttheorem}. Assume $t=3$ and $m=4$.
%which can be verified using the Macaulay2 code from \Cref{M2code}.
\begin{itemize}
\item[$i)$] If $n=5$,  then   $R/\init(I_3)$ has the WLP but fails the SLP, while $R/I_3$ has the SLP;
\item[$ii)$] if  $n=6$, then  $R/\init(I_3)$ fails  the WLP while $R/I_3$ has the SLP.  
\end{itemize} The validity of the SLP of $R/I_3$  in the above two cases can be verified using the Macaulay2 code provided in \Cref{M2code}.
Therefore, the answer to \Cref{muquestion} remains negative even in the case $m\not=n$. 
\end{examples}

We conclude this paper by discussing the sharpness of the bound provided in  \Cref{lasttheorem}. First, let us recall the following result of Wiebe \cite[Proposition 2.8]{Wi}:
\begin{proposition}[Wiebe]\label{wiebe1} 
Let  $S=K[X_1,\dots, X_N]$ be a standard graded polynomial ring and let $J$ be a homogeneous ideal of $S$. If $\gin(J)$ is the generic initial ideal
of $J$ with respect to the reverse lexicographic order, then $S/J$ has the WLP (resp. SLP) if and only if $S/\gin(J)$ has the WLP (resp. SLP).
\end{proposition}
\begin{proof}
This result was originally proved for Artinian algebras, that is, when $d=\dim S/J=0$. If $d\ge 1$, then \Cref{ginconca} together with the well-known result: $\gin(\gin(J))=\gin(J)$ (see, for example, \cite[Corollary 4.2.7]{monomialideals}),  implies that the following equalities of Hilbert functions hold for each $j\ge 0$ and for each $s\ge 1$: %and Lemma 4.3.7
%and
%\thickmuskip=0mu
\begin{align*}
\HF(S/(J,\underline\theta,L^s),j)&=\HF(S/(\gin(J), X_{N-d}^s, X_{N-d+1},\dots, X_{N}),j)\\
&=\HF(S/(\gin(\gin(J)), X_{N-d}^s, X_{N-d+1}, \dots, X_{N}),j)\\
&=\HF(S/(\gin(J),\underline\theta,L^s),j)
\end{align*}
%\thickmuskip=5mu plus 5mu
 where $\theta_1, \dots, \theta_{d}, L\in S_1$ is a sequence of general linear forms and $\underline\theta=\theta_1, \dots, \theta_{d}$. Therefore, by applying \Cref{openset}, we prove this lemma  using the same reasoning as in the proof of \cite[Proposition 2.8]{Wi}. Namely, $S/J$ has the WLP (resp. SLP) if and only if $X_{N-d}$  is a weak (resp. strong) Lefschetz element on $S/(\gin(J), X_{N-d+1}, \dots, X_{N-1})$, if and only if $S/\gin(J)$ has the the WLP (resp. SLP).
\end{proof}

%\Cref{answer} and \Cref{wiebe1} 
This result lead directly to the following consequence:
\begin{corollary}\label{gin}
Let  $S=K[X_1,\dots, X_N]$ be a standard graded polynomial ring and let $J$ be a homogeneous ideal  of $S$. Let $<$ be a monomial order on $S$ such that 
$\gin(\init_<(J))=\gin(J)$. Then $S/J$ has the WLP (resp. SLP) if and only if $S/\init_<(J)$ has the WLP (resp. SLP).
%If $\gin(J)$ is the generic initial ideal
%of $J$ with respect to the reverse lexicographic order
%
%Let $t<m=n$ and let $\tau$ be a monomial order on $R$. If $\gin(\init_\tau(I_t))=\gin(I_t)$, then $R/\init_\tau(I_t)$ has the SLP.
\end{corollary}

By running  again the Macaulay2 code from \Cref{M2code}, we obtain that 
\begin{itemize}
\item if $t=2$ and $mn\le 15$, or 
\item if $t=3$ and $m=n=4$, 
%\item if $t=3$ and  $(m,n)\in\{(4,5), (5,4)\}$, then $R/\init(I_3)$ has the WLP but fails the SLP. 
\end{itemize} 
then $R/\init(I_t)$ has the SLP. Moreover, for a specific integer $t=m-1=n-1$, one can verify using Macaulay2 computations that $\gin(\init(I_t))=\gin(I_t)$, which implies that $R/\init(I_t)$ has the SLP by \Cref{answer} and \Cref{gin}. Therefore, we propose the following question:
\begin{question}\label{conj}
Does the equality $\gin(\init(I_t))=\gin(I_t)$ hold for all $t=m-1=n-1$?
\end{question}
In particular, if this question has a positive answer, 
%If the statement of this holds for every $t$, 
then  the bound provided in  \Cref{lasttheorem} is sharp. 

%\cite[Proposition 2.8]{Wi}
%whichThis result was originally proved for Artinian algebras) also hold for ideals that are not necessarily $\mm$-primary. \Cref{casem=n} follows as a consequence of \Cref{wiebe2}.
%Moreover,  using \Cref{lasttheorem}, in \Cref{answer} we provide some counterexamples that give a negative answer to the question we posed at the beginning of this paper. 

\section{Macaulay2 code}\label{M2code}
The following Macaulay2 code is available at \url{https://github.com/hoyu26/checkLP}.
{\small
\begin{lstlisting}
--compute the first n+1 values of HF(ring I/I,-).
hilF = (I,n) -> (
    for i to n list hilbertFunction (i,I)
    )

--check whether the ideal I has the WLP/SLP.
checkLP=(I)->(
theta=ideal();
d=dim ((ring I)/I);
for i from 1 to d do theta=ideal(random(1,ring I))+theta;
J=ideal leadTerm (I+theta);
A=QQ[drop (flatten entries vars ring I, d)];
JA=sub (J,A);
a=max degree numerator reduceHilbert hilbertSeries (JA);
hf=hilF(JA,a);
print ("HF=", hf);
l=random(1,ring I);
s=0;
for v from 1 to a do (
    use ring I;
    thetal=ideal(l)^v+theta;
    ILs=ideal leadTerm (I+thetal);
    ILsA=sub(ILs,A);
    H1=hilF(ILsA,a);
    H0={1};
    if v>1 then for j from 1 to v-1 do H0=H0|{(hf#j)};
    for i from 0 to a-v do (
	h=max{0,hf#(i+v)-hf#i};
	H0=H0|{h};
	); 
    print ("s=",v);
    print ("max{0, nabla^s HF}=", H0);
    print ("the Hilbert function of the quotient by L^s 
              is", H1);
    if H0!=H1 then break;
    s=v; 
    );
if s==0 then print "The chosen linear form is not a LE for 
this ideal." else (
    if  s==a then print "This ideal has the SLP." else print 
    "This ideal has the WLP.";
    );
)

--return the determinantal ideal generated by the t-minors of 
a generic m*n matrix:
Itmn=(t,m,n)->(
    R:=QQ[x_(1,1)..x_(m,n)];
    X:=transpose genericMatrix(R,x_(1,1),n,m);
    return minors(t,X);
    )

--Example (verify the algebra QQ[x_(1,1),...,x_(4,5)/in(I_3) 
has the WLP):
checkLP ideal leadTerm Itmn(3,4,5)                                                                                                                                                                                                                                                                                                                                                                                                                                                                                                                                                                                                                                                                                                                                                                        
\end{lstlisting}
}

\bigskip

%{\bf Acknowledgments}: The author genuinely thanks Satoshi  Murai for posing an interesting question that motivated this work and for his helpful  comments and suggestions on this paper.  The author is deeply grateful to Matteo Varbaro for insightful discussions that contributed to this work, and to Aldo Conca for valuable suggestions. Sincere thanks go to 
%Alexandra Seceleanu  and Le Tuan Hoa for reading an earlier version of this manuscript and providing suggestions regarding its organization.
%%Ngo Viet Trung and Le Tuan Hoa for their guidance. 
%Special thanks are also extended to Luca Fiorindo, Francesco Strazzanti and Manolis Tsakiris for their helpful insights.
{\bf Acknowledgments}: The author genuinely thanks Satoshi  Murai for posing an interesting question that motivated this work, as well as for his helpful comments and suggestions. Deep gratitude is extended to Matteo Varbaro for insightful discussions that contributed to this work, and to Aldo Conca for providing valuable suggestions at the outset of this work, which helped guide it in the right direction. Sincere thanks are also given to Le Tuan Hoa and Alexandra Seceleanu for reading an earlier version of the manuscript and providing helpful suggestions regarding its organization. Special thanks are further extended to Luca Fiorindo, Francesco Strazzanti and Manolis Tsakiris for engaging discussions on various questions related to this work. 
Thanks are due to Matt Larson for reading the first version posted on the arXiv and pointing out a missing assumption in \Cref{rmkapp}.

%Furthermore, t
The author also acknowledges the ``Lefschetz Properties in Algebra, Geometry, Topology, and Combinatorics Preparatory School," which took place in Krak\'ow, Poland,  from May 5 to 11, 2024. In particular, heartfelt appreciation to the organizers Tomasz Szemberg and Justyna Szpond, the instructors Pedro Macias Marques, Juan Migliore and Alexandra Seceleanu, and the assistances Nancy Abdallah, Giuseppe Favacchio and Lisa Nicklasson, for creating an enriching environment in which the author gained significant insight into the Lefschetz properties. %Moreover, the author sincerely thanks .

This paper originated during the author's time at the Department of Mathematics at the University of Genoa and continued at the Vietnam Institute for Advanced Study in Mathematics. The author wishes to express gratitude to both institutions for their support. Particular thanks are due to  Ngo Viet Trung and Le Tuan Hoa for their kind hospitality during the author’s stay in Hanoi.
The author is partially supported by PRIN  2020355B8Y ``\textit{Squarefree Gr\"obner degenerations, special varieties and related topics}". %Several examples and calculations in this paper were obtained through computations using the computer algebra system Macaulay2 \cite{M2}.

\bibliographystyle{plain}
\bibliography{refs}

\end{document}